\documentclass[11pt,a4paper]{amsart}
\usepackage[all]{xy}
\usepackage{amssymb}
\usepackage{graphicx}
\usepackage{theoremref}
\usepackage{comment}
\usepackage{tabularx}
\usepackage{fancyhdr}
\usepackage{calc} 
\usepackage{graphicx}
\pdfoutput=1
\usepackage{setspace}
\usepackage[applemac]{inputenc}
\usepackage[T1]{fontenc}
\usepackage{amsmath}
\usepackage{amsfonts}

\usepackage{tikz}

\usepackage{lipsum}
\numberwithin{equation}{section}

\newcommand{\ie}{{\em i.e. }}
\newcommand{\cf}{{\em cf.}\ }

\newcommand{\ul}[1]{\underline{#1}}
\newcommand{\ol}[1]{\overline{#1}}

\newtheorem{theorem}{Theorem}[section]
\newtheorem{lemma}[theorem]{Lemma}
\newtheorem{proposition}[theorem]{Proposition}
\newtheorem{corollary}[theorem]{Corollary}

\newtheorem{lemma and definition}[theorem]{Lemma and Definition}

\theoremstyle{definition}

\newtheorem*{theorem*}{Theorem}

\newtheorem{remark}[theorem]{Remark}
\newtheorem{example}[theorem]{Example}
\newtheorem{definition}[theorem]{Definition}

\newtheorem{assumption}[theorem]{Assumption}

\newtheorem{notation}[theorem]{Notation}

\newcommand{\opname}[1]{\operatorname{\mathsf{#1}}}

\renewcommand{\mod}{\opname{mod}\nolimits}

\newcommand{\proj}{\opname{proj}\nolimits}
\newcommand{\rad}{\opname{rad}\nolimits}
\newcommand{\Mod}{\opname{Mod}\nolimits}

\newcommand{\ind}{\opname{ind}\nolimits}

\newcommand{\add}{\opname{add}\nolimits}
\newcommand{\op}{^{op}}

\newcommand{\dimv}{\underline{\dim}\,}

\newcommand{\can}{\operatorname{can}\nolimits}
\def\Trop{\operatorname{Trop}}

\newcommand{\corb}{\ \widehat{\!\! /} }

\newcommand{\pretr}{\opname{pretr}\nolimits}

\newcommand{\ac}{\mathcal{A}c}

\newcommand{\eoe}{\ce \corb E}

\newcommand{\ccf}{\cc \corb F}

\newcommand{\eaf}{\ce \corb F}
\newcommand{\taf}{\widetilde{\ct} \corb F}
\newcommand{\ctt}{\widetilde{\ct}}
\newcommand{\eafast}{\ce \corb F_*}

\newcommand{\colim}{\opname{colim}}
\newcommand{\cok}{\opname{cok}\nolimits}

\newcommand{\res}{\opname{res}}

\newcommand{\Z}{\mathbb{Z}}

\newcommand{\T}{\mathbb{T}}

\renewcommand{\P}{\mathbb{P}}

\newcommand{\iso}{\stackrel{_\sim}{\rightarrow}}

\newcommand{\Hom}{\opname{Hom}}

\newcommand{\Ext}{\opname{Ext}}

\newcommand{\ten}{\otimes}
\newcommand{\lten}{\overset{\boldmath{L}}{\ten}}

\newcommand{\gpr}{\operatorname{gpr}\nolimits}

\newcommand{\ca}{{\mathcal A}}
\newcommand{\cb}{{\mathcal B}}
\newcommand{\cc}{{\mathcal C}}
\newcommand{\cd}{{\mathcal D}}
\newcommand{\ce}{{\mathcal E}}

\newcommand{\ch}{{\mathcal H}}

\newcommand{\cp}{{\mathcal P}}
\newcommand{\cR}{{\mathcal R}}

\newcommand{\cs}{{\mathcal S}}
\newcommand{\ct}{{\mathcal T}}

\newcommand{\cx}{{\mathcal X}}

\begin{document}
\title{A 2-Calabi-Yau realization of finite-type cluster algebras
 with universal coefficients\\
 }
\author{ALFREDO N\'AJERA CH\'AVEZ}
\address {CONACyT--Instituto de matemáticas UNAM\\
       Le\'on 2, altos, Oaxaca de Ju\'arez \\
                  68000 Oaxaca\\
                  M\'exico.}
\email {najera@matem.unam.mx}

\maketitle

\begin{abstract}
We categorify various finite-type cluster algebras with coefficients using completed orbit categories associated to Frobenius categories. Namely, the Frobenius categories we consider are the categories of finitely generated Gorenstein projective modules over the singular Nakajima category associated to a Dynkin diagram and their standard Frobenius quotients. In particular, we are able to categorify all finite-type skew-symmetric cluster algebras with universal coefficients and finite-type Grassmannian cluster algebras. Along the way, we classify the standard Frobenius models of a certain family of triangulated orbit categories which include all finite-type $n$-cluster categories, for all integers $n\geq 1$.
\end{abstract}

\section{Introduction}
In this article we pursue the representation-theoretic approach to categorification of cluster algebras that has been developed by many authors (see for example \cite{Amiot,BMRRT,Demonet Iyama, Demonet Luo,Fu Keller,GLS flag, JKS}). Here, we are able to categorify a new family of skew-symmetric finite-type cluster algebras with geometric coefficients and give a new description of the categories constructed in \cite{JKS} to categorify finite-type Grassmannian cluster algebras (note that in loc. cit. the authors categorify \emph{all} Grassmannian cluster algebras). Our main idea is to combine the techniques used to categorify acyclic cluster algebras with those used to categorify geometric cluster algebras.

Cluster categories were introduced in \cite{BMRRT} to study \emph{coefficient free} cluster algebras associated to acyclic quivers using representation theory. They are defined as follows. Let $k$ be an algebraically closed field and $Q$ a finite quiver without oriented cycles. The cluster category $\cc_Q$ is defined as the orbit category
\begin{equation*}
\cd^b(\mod (kQ)) / \Sigma \circ \tau^{-1},
\end{equation*}
where $\cd^b(\mod (kQ))$ is the bounded derived category of finitely generated right modules over the path algebra $kQ$ and $\Sigma \circ \tau^{-1}$ is the composition of the suspension functor $\Sigma$ and the inverse of the Auslander-Reiten translation $\tau^{-1}$. Some of the fundamental properties of cluster categories are as follows:
\begin{itemize}
\item they are $2$-Calabi-Yau triangulated categories (2-CY for short);
\item they Krull-Schmidt categories;
\item they have a cluster-tilting subcategory.
\end{itemize}
It was noticed in \cite{BIRS} that Frobenius categories can provide an adequate framework to categorify cluster algebras with geometric coefficients. This proposal has been successfully implemented by many authors. For instances we can mention the work of Fu and Keller \cite{Fu Keller}; Geiss-Leclerc-Schr\"oer's catgorification of the multi-homogenus coordinate ring of a partial flag variety \cite{GLS flag}; the categorification of Grassmannian cluster algebras obtained by Jensen, King and Su in \cite{JKS}; the recent generalization of \cite{GLS flag} and \cite{JKS} by Demonet and Iyama \cite{Demonet Iyama}; the categorification of cluster algebras associated to ice quivers with potentials arising from triangulated surfaces carried out by Demonet and Luo in \cite{Demonet Luo}, among many other works. 

In this paper we combine the approaches of \cite{BMRRT} and \cite{BIRS} to categoryfiy cluster algebras with geometric coefficients using orbit categories associated to Frobenius categories. To be more precise, let $\ce$ be a $k$-linear Frobenius category whose stable category is triangle equivalent to $\cd^b(\mod (kQ))$. Suppose that the autoequivalence $\Sigma \circ \tau^{-1}$ can be lifted to an exact autoequivalence $E:\ce \overset{\sim}\to \ce$. It follows from previous work of the author of this note \cite{Najera Frobenius orbit} that the \emph{completed orbit category} $ \eoe $ has the structure of a Frobenius category, whose stable category is triangle equivalent to the cluster category $\cc_Q$. It is reasonable to expect that the Frobenius category $\eoe$ can be used to construct a categorification of a geometric cluster algebra of type $Q$. The main purpose of this note is to show that this is indeed the case provided $Q$ is a Dynkin quiver and $\ce$ is a standard Frobenius model of $\cd^b(\mod (kQ))$ in the sense of \cite{Keller Scherotzke 1}, that is, $\ce$ is Hom-finite, Krull-Schmidt, and satisfies the following conditions:
\begin{itemize}
\item[(i)] For each indecomposable projective object $P$ of $\ce$, 
the $\ce$-module $\rad_\ce(?,P)$ and the $\ce^{op}$-module
$\rad_\ce(P,?)$ are finitely generated with simple tops.
\item[(ii)] $\ce$ is standard in the sense of Ringel \cite{Ringel84}, \ie its category of indecomposable objects is equivalent to
the mesh category of its Auslander--Reiten quiver. 
\end{itemize}
So let's assume from now on that $Q$ is an orientation of a Dynkin diagram $\Delta$ of type $A$, $D$ or $E$. In general, different Frobenius models will lead to the categorification of different cluster algebras of type $Q$ with coefficients. An important aspect of our work is that we are able to identify particular Frobenius models of $\cd^b(\mod (kQ))$ which provide a categorification of distinguished cluster algebras such as cluster algebras with universal coefficients and finite-type Grassmannian cluster algebras (the later happens when $\Delta$ is $D_4$, $E_6$, $E_8$ or of type $A$). 

The standard Frobenius models of $\cd^b(\mod (kQ))$ were classified by Keller and Scherotzke in \cite{Keller Scherotzke 1} using a representation-theoretic approach to Nakajima's quiver varieties. They are in bijection with certain subsets of the set of vertices of the repetition quiver $\Z Q$ and can be described as follows. Let $\cs$ be the singular Nakajima category (see \cite{Keller Scherotzke 1}) associated to $Q$ and $C \subset \Z Q_0$ an admissible configuration. Let $\cs_C$ be the quotient category of $\cs$ associated to $C$ constructed in \cite{Keller Scherotzke 1} (\cf \thref{admissible configuration}). Then the category of finitely generated Gorenstein projective $\cs_C$-modules  $\gpr (\cs_C)$ is a standard Frobenius model of $\cd^b(\mod (kQ))$. Moreover, every standard Frobenius model of $\cd^b(\mod (kQ))$ is equivalent as an exact category to $\gpr (\cs_C)$ for  some admissible configuration $C$. The main result of this paper is as follows:

\begin{theorem*}
Consider $\Z Q_0$ as the set of indecomposable objects of $\cd^b(\mod (kQ))$. Let $C$ be an admissible configuration of $\Z Q_0$ invariant under $\Sigma \circ \tau^{-1}$. Then $\Sigma \circ \tau^{-1}$ lifts to an exact automorphism $E: \gpr (\cs_C) \to \gpr(\cs_C)$ and the complete orbit category $ \gpr(\cs_C) \corb E$ is a 2-Calabi-Yau realization 
of a cluster algebra with geometric coefficients of type $\Delta$. Moreover, if $C =\Z Q_0$ $($\ie $\cs_C = \cs)$ then  $ \gpr(\cs)\corb E$ is a 2-Calabi-Yau realization of the cluster algebra with \emph{universal coefficients} of type $\Delta$. 
\end{theorem*}

Cluster algebras with universal coefficients were introduced by Fomin and Zelevinsky in \cite{clusters 4} and further investigated by Reading in \cite{Reading universal, Reading surfaces,Reading Markov}. These are cluster algebras which are universal with respect to coefficient specialization among cluster algebras associated to a fixed initial quiver. The existence of universal coefficients for finite-type cluster algebras was proved in \cite{clusters 4}. In \cite{Reading universal}, it was shown that universal coefficients always exist if we restrict to the class of geometric cluster algebras and allow them to have an infinite number of coefficients (whose powers can be taken not only in the ring of integer numbers, but in more general rings such as the rational or real numbers). We hope that our approach can be generalized to categorify other cluster algebras with universal coefficients beyond finite-type.

Part of the technical aspects of our construction were carried out in \cite{Najera Frobenius orbit}, where we studied completed orbit categories associated to Frobenius categories in a more general framework. It is worth pointing out that we consider completed orbit categories rather than usual orbit categories because the Krull-Schmidt property in general is lost when we consider orbit categories, whereas it is always preserved for completed orbit categories.

We can use the insight of \cite{Keller Scherotzke 1} to classify the Frobenius models of $\cc_Q$ satisfying conditions (i) and (ii) using completed orbit categories associated to categories of the form $\gpr(\cs_C)$. Notice that we shall modify slightly the definition in \cite{Keller Scherotzke 1} and admit Frobenius models with infinite-dimensional morphism spaces. 
Moreover, we can classify the Frobenius models not only of cluster categories but of a larger class of orbit categories associated to $\cd^b(\mod(kQ))$. This classification problem was already addressed by Scherotzke in \cite[Section 3]{Scherotzke self-injective} using usual orbit categories. The following theorem proved in this note extends Theorem 3.7 of \cite{Scherotzke self-injective} to more general set-ups.

\begin{theorem*}
Let $E:\cd^b (\mod (kQ)) \to \cd^b (\mod (kQ))$ be a triangle equivalence such that $\cd^b(\mod (kQ))/E $ is $\Hom$-finite and equivalent to its triangulated hull (in the sense of \cite{Keller triang}). Let $C$ be an admissible configuration invariant under $E$. Suppose moreover that $E$ lifts to an exact autoequivalence $E_*:\gpr (\cs_C)\to \gpr(\cs_C)$, and that for each indecomposable object $X$ of $\gpr(\cs_C)$, the group $\gpr(\cs_C)(X, E_*^l(X))$ vanishes for all $l<0$. Then
\begin{itemize}
\item[$(a)$] The completed orbit category $\gpr(\cs_C) \corb E_{\ast}$ admits the structure of a Frobenius category whose stable category is triangle equivalent to $\cd_Q/E$.
\item[$(b)$] The category $\gpr(\cs_C)\corb E_{\ast}$ satisfies conditions (i) and (ii) above. 
\item[$(c)$] The map taking $C$ to $\gpr(\cs_C)$ induces a bijection from the set of $E$-invariant admissible configurations $C\subset \Z Q_0$ onto the set of equivalence classes of Frobenius models $\gpr(\cs_C)$ of $\cd_Q/E$
satisfying (i) and (ii).
\end{itemize}
\end{theorem*}

The main difference between \cite[Theorem 5.7]{Scherotzke self-injective} and the theorem above is that, by considering \emph{completed} orbit categories, we are able to consider functors $E:\cd^b (\mod (kQ)) \to \cd^b (\mod (kQ))$ and admissible configurations $C\subset \Z Q_0$ for which the usual orbit category $\gpr(\cs_C) / E_*$ is $\Hom$-infinite or fails to be Krull-Schmidt. 

This paper is organized as follows. In Section 2 we recall the definition and some of the fundamental properties of Nakajima categories. In Section 3 we state our main results and prove them in Sections 4 and 5. In Section 6 we present some examples.


\section{Recollections} 

In this note we will freely use the basic concepts in the theory of dg categories. The reader can refer to Section $2$ of \cite{Najera Frobenius orbit} for an account of the results on dg categories and their orbit categories that will be used here. Some of the standard references for this topic are \cite{Keller on dg} and \cite{Drinfeld}. Throughout this note all functors between $k$-linear categories are assumed to be $k$-linear. 

\subsection{Notation}
The set of morphisms between two objects $x$ and $y$ of a category $\cc$ is denoted by $\cc(x,y)$. If $k$ is a field and $\ce$ an additive $k$-category, a right $\cc$-module is by definition a $k$-linear functor $M:\cc^{\text{op}} \rightarrow \Mod (k)$, where $\Mod (k)$ is the category of $k$-vector spaces. We let $\Mod(\cc)$ be the category of all right $\cc$-modules, 
$\mod (\cc)$ its subcategory of finitely presented modules 
and  $\proj(\cc)$ the full subcategory of finitely generated projective $\cc$-modules. In particular, for each object $x$ of $\cc$, we have the finitely generated projective $\cc$-module
\begin{equation*}
x^{\wedge}:=\cc(?,x):\cc^{\rm op} \to \Mod (k),
\end{equation*}
and the finitely generated injective $\cc$-module 
\[
x^{\vee}:=D(\cc(x,?)):\cc^{\rm op} \to \Mod (k).
\]
Here, $\cc (y,z)$ denotes the space of morphisms from $y$ to $z$ in the category $\cc$ and $D$ is the duality over the ground field $k$. 
If the endomorphism ring of every indecomposable object $x$ of $\cc$ is local, then each projective $\cc$-module is a direct sum of modules of the form $x^{\wedge}$, therefore we will sometimes refer to $x^{\wedge}$ (resp. $x^{\vee}$) as the free (resp. co-free) $\cc$-module associated to $x$. 
Contrary to the notation used for the space of morphisms in $\cc$, the morphism space between two $\cc$-modules $L$ and $M$ is denoted by $\Hom_{\cc}(L,M)$ or simply $\Hom (L,M)$ when there is no risk of confusion.

\subsection{Nakajima categories} In this section we recall the definition and some of the properties of the Nakajima categories.

Let $Q$ be a quiver. Let $Q_0$ be its set of vertices and $Q_1$ be its set of arrows. We suppose that $Q$ is finite (both $Q_0$ and $Q_1$ are finite sets) and acyclic ($Q$ has no oriented cycles). The \emph{repetition quiver} (\cf \cite{Riedtmann80b}) 
$\Z Q$ is the quiver obtained from $Q$ as follows:
\begin{itemize}
\item the set of vertices of $\Z Q$ is $\Z Q_0= Q_0 \times \Z$.
\item For each arrow $\alpha : i\longrightarrow j$ of $Q$ and each $p\in \Z$, the repetition quiver $\Z Q$ has the arrows
\begin{equation*}
\xymatrix{
(\alpha,p):(i,p) \ar[r]  & (j,p) & \text{and} &  \sigma(\alpha, p):(j,p-1) \ar[r] & (i,p).
}
\end{equation*}
\item $\Z Q$ has no more arrows than those described above. 
\end{itemize}
Let $\sigma:\Z Q_1\to \Z Q_1$ be the bijection given by
\begin{equation*}
\sigma (\beta)=
\begin{cases}
\sigma(\alpha,p) & \text{if  } \beta=(\alpha,p), \\
(\alpha, p-1) & \text{if  } \beta=\sigma(\alpha,p).
\end{cases}
\end{equation*}
Let $\tau : \Z Q \to \Z Q$ be the graph automorphism given by \emph{the translation by one unit}:
\begin{equation*}
\xymatrix{
\tau(i,p)=(i,p-1)    & \text{and} & \tau(\beta)=\sigma^2(\beta) 
}
\end{equation*}
for each vertex $(i,p)$ and each arrow $\beta$ of $\Z Q$. 

Let $k$ be a field. Following \cite{Gabriel80} and \cite{Riedtmann80}, we define the {\em mesh category $k(\Z Q)$} to be the quotient of the path category $k\Z Q$ by the ideal generated by the mesh relators, \ie the $k$-category whose objects are the vertices of $\Z Q$ and whose morphism space from $a$ to $b$ is the space of all $k$-linear combinations of paths from $a$ to $b$ modulo the subspace spanned by all elements $u r_x v$, where $u$ and $v$ are paths and 
\[
r_x = \sum_{\beta: y \to x} \beta \sigma(\beta): \quad
\raisebox{1.225cm}{\xymatrix@R=0.5cm@C=0.5cm{  & y_1 \ar[dr]^{\beta_1} & \\
\tau(x) \ar[ur]^{\sigma(\beta_1)} \ar[dr]_{\sigma(\beta_s)}  & \vdots & x \\
 & y_s \ar[ur]_{\beta_s} & }}
\]
is the {\em mesh relator} associated with a vertex $x$ of $\Z Q$. Here the sum runs over all arrows $\beta: y \to x$ of $\Z Q$. Note that $\tau$ can be thought of as an autoequivalence of $\Z (kQ)$.

\begin{notation}
Let $kQ$ be the path algebra of $Q$ and let $\mod (kQ)$ be the category of all finite-dimensional right $kQ$-modules. We adopt the notation used in \cite{Keller Scherotzke 1} and denote the bounded derived category of $\mod (kQ)$ by $\cd_Q$. 
If $\cx$ is an additive category then we write $\ind(\cx)$ to denote its full subcategory 
of indecomposable objects. 
\end{notation}

\begin{theorem}[\cite{Happel derived}] 
\thlabel{Happel embedding}
There is a canonical fully faithful functor
\[
H: k(\Z Q) \to \ind(\cd_Q) 
\]
taking each vertex $(i,0)$ to the indecomposable
projective module associated to the vertex $i\in Q_0$. It is an equivalence if and only if
$Q$ is a Dynkin quiver (= an orientation of a Dynkin diagram of type ADE).
\end{theorem}

\begin{definition}
(\cite{Keller Scherotzke 1})
The \emph{framed quiver} $\widetilde{Q}$ associated to $Q$ is the quiver obtained from $Q$ by adding, for each vertex $i\in Q_0$, a new vertex $i'$ and a new arrow $i\to i'$. We consider the repetition quiver $\Z \widetilde{Q}$ and call  \emph{frozen vertices} its vertices of the form $(i',n)$, with $i\in Q_0$ and $n\in \Z$. The \emph{regular Nakajima category} $\cR$ associated to $Q$ is the quotient of the path category $k\Z \widetilde{Q}$ by the ideal generated by the mesh relators associated to the \emph{non-frozen vertices}. The \emph{singular Nakajima category} $\cs$ is the full subcategory of $\mathcal{R}$ whose objects are the frozen vertices. 
\end{definition} 

As shown in \cite{Happel derived}, $H$ identifies the autoequivalence $\tau: \Z (k Q) \to \Z (k Q)$ with the {\em Auslander-Reiten translation} of  $\cd_Q$, 
which we will also denote by $\tau$. For Dynkin quivers,
the combinatorial descriptions of $\nu$, $\Sigma$ and of the image of 
$\mod (kQ)$ in $\cd_Q$ are given in section~6.5 of
\cite{Gabriel80}. If $Q$ is Dynkin, let $\Sigma$ be the unique bijection of the vertices of $\Z Q$ such that
\begin{equation*}
H(\Sigma x)= \Sigma(H(x)).
\end{equation*}

\begin{remark}
\thlabel{identification}
We will systematically identify the vertices of $\Z Q$ with the objects of $\ind(\cd_Q)$ via the functor $H$.
So when we refer to a vertex of $\Z Q$ as an indecomposable object of $\cd_Q$ we will tacitly refer to the object $H(x)$.
\end{remark}

\begin{remark}
Note that $\Sigma:\Z Q_0 \to \Z Q_0$ extends to a bijection $\Z \widetilde{Q}_0 \to \Z \widetilde{Q}_0$ which will be denoted by $\Sigma$ by a slight abuse of notation. There is also a bijection $\sigma: \Z \widetilde{Q}_0\to \Z \widetilde{Q}_0$ given by $\sigma: (i,n)\mapsto (i',n-1)$ and $(i',n)\mapsto (i,n)$ for $i$ a vertex of $Q$ and $n$ an integer. 
\end{remark}

\begin{example}
Let $Q$ be the Dynkin quiver $1\to 2$. The quiver $\Z \widetilde{Q}$ is depicted in Figure \ref{cR type A_2} below. The frozen vertices are represented by small squares $\Box$. In this case the mesh relations imply that $ba+dc=0$ in the Nakajima category whereas $eb\neq 0$. 
\begin{figure}[htbp]
\[
\xymatrix{
\Box \ar[r] & \ar[r] \ar[rd] & \Box \ar[r] & \ar[r] \ar^d[rd] & \Box \ar[r] & \ar[r] \Sigma x \ar[rd] & \Box \\
\tau (x) \ar[r] \ar[ru] & \ar[r] \Box & \ar^a[r] x \ar^c[ru] & \Box \ar^b[r] & \ar^e[r] \ar[ru] &  \Box \ar[r] &
}
\]
\caption{The quiver of the regular Nakajima category associated to $A_2$.}
\label{cR type A_2}
\end{figure}
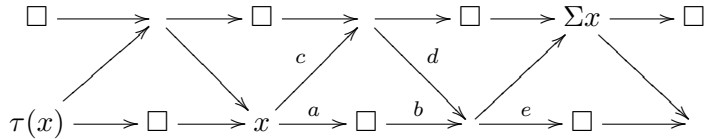
\end{example}

\begin{definition}
\thlabel{admissible configuration}
Let $C$ be a subset of $\Z Q_0$. Denote by $\cR_C$ the quotient of $\cR$ by the ideal generated by the identities of the frozen vertices not belonging to $\sigma^{-1}(C)$ and by $\cs_C$ its full subcategory formed by the vertices in $\sigma^{-1}(C)$. We call $C$ an \emph{admissible configuration of} $\Z Q_0$ if for each vertex $x \in \Z Q_0$, there is a vertex $c$ in $C$ such that the space of morphisms from $x$ to $c$ in the mesh category $k(\Z Q)$ does not vanish. Finally, let $\Z \widetilde{Q}_C$ be the quiver obtained from $\Z \widetilde{Q}$ by delating the set of frozen vertices $\sigma^{-1}(C)$.
\end{definition}

\subsection{Nakajima categories and derived categories}

A {\em Frobenius category} is a Quillen-exact category  \cite{Quillen} with enough injective objects, enough projective objects and where these two families of objects coincide. By definition,
every Frobenius category is endowed with a distinguished class of sequences
\[
0 \to L \to M \to N \to 0
\]
called \emph{conflations} (we follow the terminology of \cite{Gabriel Roiter}). We will call the morphism $L\to M$ of such a sequence an \emph{inflation} and the morphism $ M \to N $ a \emph{deflation}. Let $\ce$ be a Frobenius category. The \emph{stable category} $\ul{\ce}$ is the quotient of $\ce$ by the ideal of morphisms factoring through  a projective-injective object. It was shown by Happel \cite{Happel} that $\ul{\ce}$ has a canonical
structure of triangulated category. Note that it is possible to define the extension functors $\Ext^i_\ce$ for $\ce$ in the usual way. We have that 
\[
\Ext^i_\ce(L,M) \iso \Ext^i_{\ul{\ce}}(L,M)
\]
for all objects $L$ and $M$ of $\ce$ and all integers $i\geq 1$. Now we introduce a class of Frobenius categories that will be very important for our discussion.

\begin{definition}
\thlabel{gpr}
Let $\cc$ be an additive $k$-category. An $\cc$-module $M$ is \emph{finitely generated Gorenstein projective} if there is an acyclic complex
\begin{equation*}
P_M: \cdots \to P_1 \to P_0 \to P^{0} \to P^{1} \to \cdots
\end{equation*}
of objects in $\proj(\cc)$ such that $M \cong \cok (P_1\to P_0)$ and the complex $\Hom_{\cc}(P_M, P')$ is still acyclic for each module $P'$ in $\proj(\cc)$. Denote by $\gpr (\cc)$ the full subcategory of $\mod (\cc)$ formed by the Gorenstein projective modules. In the situation described above we call $P_M$ a \emph{complete projective resolution} of $M$.
\end{definition}
Every finitely generated projective $\cc$-module $P$ lies in $\gpr(\cc)$ since we may consider the complete projective resolution $\cdots \to 0 \to P \overset{\sim}{\longrightarrow}  P \to 0 \to \cdots$. The following result is a well-know result on Gorenstein projective modules (see \cite[Proposition 5.1]{AR}).
\begin{lemma}
\thlabel{gpr is Frobenius}
The category $\gpr(\cc)$ is an extension closed subcategory of $\Mod (\cc)$. Moreover, the induced exact structure on $\gpr (\cc)$ makes it a Frobenius category whose subcategory of projective--injective objects is $\proj(\cc)$.
\end{lemma}

\begin{remark}
The Gorenstein injective modules are defined analogously, all the results in this note can be naturally adapted to be stated in terms of Gorenstein injective modules.
\end{remark}

The following theorem summarizes some of the results of \cite{Keller Scherotzke 1} that we shall need.

\begin{theorem} $($\cite{Keller Scherotzke 1}$)$
\thlabel{description of gpr}
Let $Q$ be a Dynkin quiver and $C$ an admissible configuration of $\Z Q$. Then 
\begin{itemize}
\item[$(i)$] the restriction functor 
\begin{equation*}
\res: \Mod (\cR_C)\to \Mod (\cs_C)
\end{equation*}
induces an equivalence between the full subcategory of finitely generated projective $\cR_C$-modules $\proj (\cR_C)$ and the category $\gpr(\cs_C)$
\begin{equation*}
\res:\proj(\cR_C)\overset{\sim}{\to} \gpr(\cs_C). 
\end{equation*}
In particular, it induces an isomorphism of $\Z \widetilde{Q}_C$ onto the Auslander-Reiten quiver of $\gpr(\cs_C)$ so that the vertices of $\sigma^{-1}(C)$ correspond to the projective--injective objects;
\item[$(ii)$] there is a $\delta$-functor $\Phi:\mod (\cs_C)\to \cd_Q$ defined as the composition
\begin{equation*}
\mod(\cs_C)\overset{\Omega}{\longrightarrow} \ul{\gpr}(\cs_C) \overset{\phi}{\longrightarrow} \cd_Q,
\end{equation*}
where $\Omega$ is the syzygy functor and $\phi$ is a distinguished triangle equivalence.
\end{itemize}
\end{theorem}

\begin{remark}
\thlabel{indecomposables of gpr}
In view of part $(i)$ of \thref{description of gpr}, 
the indecomposable objects of $\gpr(\cs_C)$ are of the form 
$\res(x^{\wedge})$, for a vertex $x$ of $\Z \widetilde{Q}_C$. 
If $y$ is a frozen vertex of $\Z \widetilde{Q}_C$, then the projective $\cs_C$-module $y^{\wedge}\in \gpr(\cs_C)$ is identified with $\res(y^{\wedge})$. For this reason, we denote the indecomposable objects of $\gpr(\cs_C)$ by $x^{\wedge}$, for some vertex $x$ of $\Z \widetilde{Q}_C$. 
\end{remark}

\begin{assumption}
From now on we suppose that $Q$ is an orientation of a connected and simply laced Dynkin diagram.
\end{assumption}

\subsection{Standard Frobenius models of $\cd_Q$}
\label{section Frobenius models}

We suppose for the rest of the note that $k$ is algebraically closed. A {\em Frobenius model of $\cd_Q$} is
a Frobenius category $\ce$ together with a triangle equivalence
$\cd_Q \cong \ul{\ce}$. \thref{description of gpr} implies that we can construct Frobenius models of $\cd_Q$ from admisible configurations. Indeed, if $C \subset \Z Q_0$ is admissible, then the category $\gpr(\cs_C)$ is a Frobenius model of $\cd_Q$. Keller and Sharotzke proved that the Frobenius categories obtained in this way are precisely the class of Frobenius categories $\ce$ satisfying the following conditions:
\begin{itemize}
\item [(P0)] $\ce$ is $k$-linear, $\Ext$-finite and Krull-Schmidt (see \cite{Krause Krull-Schmidt}).
\item[(P1)] For each indecomposable non-projective object $X$ of $\ce$, there is an almost split sequence starting and an almost split sequence ending at $X$.
\item[(P2)] For each indecomposable projective object $P$ of $\ce$, 
the $\ce$-module $\rad_\ce(?,P)$ and the $\ce^{op}$-module
$\rad_\ce(P,?)$ are finitely generated with simple tops.
\item[(P3)] $\ce$ is standard, i.e. its category of indecomposables is equivalent to
the mesh category of its Auslander--Reiten quiver
(\cf~section~2.3, page~63 of \cite{Ringel84}).
\end{itemize}

\begin{theorem} $($\cite[Corollary 5.25]{Keller Scherotzke 1}$)$
\thlabel{Frobenius models}
The map taking $C$ to $\gpr(\cs_C)$ induces a bijection from the
set of admissible configurations $C\subset \Z Q_0$ onto the set of
equivalence classes of Frobenius models $\ce$ of $\cd_Q$
satisfying $(P0)-(P3)$. The inverse bijection sends a Frobenius
model $\ce$ to the set $C\subset \Z Q_0$ such that the indecomposable
projectives of $\ce$ correspond to the vertices $\sigma^{-1}(c)$,
$c\in C$, of the Auslander--Reiten quiver of $\ce$.
\end{theorem}

\subsection{Completed orbit categories}
Orbit categories and their completions will be fundamental in what follows. We recall now their definitions and useful properties. Let $\cc$ be a $k$-linear category and $F: \cc \to \cc$ an automorphism. By definition, the orbit category $\cc/F$ has the same objects as $\cc$, the set of morphisms from an object $X$ to an object $Y$ is given by 
\begin{equation*}
\label{usual orbit category}
\cc / F(X,Y)= \bigoplus_{l\in \Z} \cc(X,F^l(Y)).
\end{equation*}	
The composition of morphisms is given by the formula
\begin{equation}
\label{composition}
(f_a)\circ(g_b)=\left(\sum_{a+b=c}F^b(f_a)\circ g_b\right),
\end{equation} 
where $f_a:Y\to F^a(Z)$, $g_b:X\to F^b(Y)$ and $a,b\in \Z$. 
Clearly $\cc/F$ is still a $k$-linear category and the canonical projection $p:\cc \to \cc/F$ is an additive functor. Suppose that for all objects $X$, $Y$ of $\cc$, the space $\cc (X,F^l (Y))$ vanishes for all integers $l\ll 0$. In his case we can define the \emph{completed orbit category} $\ccf$ as the category whose objects are the same as those of $\cc$ and with morphism spaces
\begin{equation}
\cc \corb F(X,Y)= \prod_{l\in \Z} \cc(X,F^l(Y)).
\end{equation}
The vanishing condition imposed on the spaces $\cc (X,F^l (Y))$ implies thst (\ref{composition}) defines a composition of morphisms in $\ccf$. Clearly, the category $\cc \corb F$ is $k$-linear. We still denote the natural projection $ \cc \to \ccf$ by $p$.

\begin{assumption}
\thlabel{remark on completed orbit categories}
Whenever we refer to the completed orbit category associated to an automorphism $F:\cc\to \cc$ we will implicitly assume for all objects $X$, $Y$ of $\cc$, the space $\cc (X,F^l (Y))$ vanishes for all integers $l\ll 0$.
\end{assumption}

\begin{remark}
Using a standard procedure, we can replace a category with an autoequivalence by a category with an automorphism (see Section 7 of \cite{Asashiba 2}). So we will consider orbit categories associated to categories equipped with an autoequivalence.
\end{remark}

\begin{remark}
\thlabel{indecomposables orbit}
Each indecomposable object of $\ccf$ is the image of an indecomposable object of $\cc$ under $p$.
\end{remark}

Let $\cb$ be a dg category endowed with an endomorphism $F:\cb\to \cb$ inducing an equivalence $H^0(F):H^0(\cb)\to H^0(\cb)$. We define the \emph{dg orbit category} $\cb/F$ as follows: the objects of $\cb/F$ are the same as the objects of $\cb$. For $X, Y\in \cb/F$, we have
\begin{equation}
\cb / F(X,Y):= \colim_p\bigoplus_{n \geq 0}\cb (F^n(X),F^p(Y)),
\end{equation}
where the transition maps are given by $F$

\begin{equation*}
\xymatrix{
\displaystyle{\bigoplus_{n\geq 0}}\cb (F^n(X),F^p(Y))\ar^{F}[r] & \displaystyle{\bigoplus_{n\geq 0}}\cb (F^n(X),F^{p+1}(Y)).
}
\end{equation*}

\begin{definition}
\thlabel{triangulated hull}
Let $\ct$ be a triangulated category endowed with a triangulated equivalence $F:\ct \to \ct $. Suppose that $\ct_{dg}$ is a dg enhancement of $\ct$ and that $\widetilde{F}:\ct_{dg} \to \ct_{dg}$ is a dg functor such that $H^0(\widetilde{F})=F$.
The triangulated hull of $\ct/F$ (with respect to $\widetilde{F}$) is the triangulated category $H^0(\pretr({\ct_{dg}/\widetilde{F}}))$, 
where $\pretr({\ct_{dg}/\widetilde{F}})$ is the pretriangulated hull of $\ct_{dg}/\widetilde{F}$. 
\end{definition}

\section{The main results}
\label{section main results}

This section is devoted to state our main results, their proofs will be given in the subsequent sections. Let $F:\cd_Q\to \cd_Q$ be a triangle equivalence which is isomorphic to the derived tensor product

\begin{equation*}
?\lten_{kQ}M:\cd_Q \to \cd_Q
\end{equation*} 
for some complex $M$ of $kQ$-bimodules. This is not a restriction since all autoequivalences with an "algebraic" construction are of this form (\cf Section 9 of \cite{Keller triang}). 

We can identify $F$ with an automorphism of the mesh category $k(\Z Q)$ via Happel's embedding (see \thref{Happel embedding}). A configuration $C \subset \Z Q_0$ is called $F$-invariant if $F(C)=C$. Note that $F$ can be extended to an automorphism of $\cR_C$ provided $C$ is admissible and $F$-invariant. In this case, we still denote by $F$ the extension of $F:k(\Z Q) \to k(\Z Q) $ to $\cR_C$. Let $F_*$ be the automorphism on $\Mod (\cR_C)$ defined by 
\begin{equation*}
F_{\ast}(M)= M \circ F^{-1}. 
\end{equation*}
In particular, $F_{\ast}(x^{\wedge})=F(x)^\wedge$ for every vertex $x$ of $\Z Q_C$, so $F_*$ restricts to an exact automorphism $F_*:\gpr(\cs_C) \to \gpr(\cs_C)$.

Consider the dg functor
\begin{equation}
\label{canonical dg lift of F}
\widetilde{F}:={?\otimes_{kQ} P^{\bullet} M}: C^b(\proj(kQ))_{dg} \to  C^b(\proj(kQ))_{dg}
\end{equation}
where $P^{\bullet} M$ is a projective resolution of the bimodule $M$. We have that $H^0(\widetilde{F})=F$.

\begin{theorem}
\thlabel{Frobenius model orbit}
Let $F:\cd_Q \to \cd_Q$ be a triangle equivalence such that $\cd_Q/F $ is $\Hom$-finite and equivalent to its triangulated hull. Let $C$ be an admissible configuration invariant under $F$. Suppose moreover that for each indecomposable object $X$ of $\gpr(\cs_C)$ the group $\Hom_{\cs_C}(X, F^l_*(X))$ vanishes for all $l<0$. Then
\begin{itemize}
\item[$(i)$] the completed orbit category $\gpr(\cs_C) \corb F_{\ast}$ admits the structure of a Frobenius category whose stable category is triangle equivalent to $\cd_Q/F$ $($see \thref{remark on completed orbit categories}$)$;
\item[$(ii)$] the category $\gpr(\cs_C) \corb F_*$ satisfies conditions $(P_0)-(P_3)$ of Section \ref{section Frobenius models} and its AR quiver is isomorphic to $\Z \widetilde{Q}_C/F$;
\item[$(iii)$] the map taking $C$ to $\gpr(\cs_C)$ induces a bijection from the set of $F$-invariant admissible configurations $C\subset \Z Q_0$ onto the set of equivalence classes of Frobenius models of $\cd_Q/F$
satisfying $(P0)-(P3)$.
\end{itemize}
\end{theorem}

We can see that in the setting of \thref{Frobenius model orbit}, the completed orbit category $\gpr(\cs_C) \corb F_{\ast}$ is in general $\Hom$-infinite. Note however that it is always $\Ext^1$-finite since $\cd_Q/F$ is $\Hom$-finite. Keller showed in \cite{Keller triang} that $\cd_Q/F$ is $\Hom$-finite and equivalent to its triangulated hull if the following conditions hold:
\begin{itemize}
\item[$(a)$] for each indecomposable $U$ of $\mod (kQ)$, there are only finitely many $i\in \Z$ such that the object $F^i (U)$ lies in $\mod (kQ)$;
\item[$(b)$] there is an integer $N \geq 0$ such that the $F$-orbit of each indecomposable of $\cd_Q$ contains an object $\Sigma^nU$, for some $0 \leq n \leq N$ and some indecomposable object $U$ of $\mod (kQ)$.
\end{itemize}	

\begin{remark}
\thlabel{higher cluster categories}
The functor $\Sigma^n\circ \tau^{-1}$ satisfies these conditions for all $n\geq 0$ (we are in the Dynkin case, so $n=0$ is included). These are the functors used to define higher cluster categories, so \thref{Frobenius model orbit} holds for these kind of categories.
\end{remark}

\begin{theorem}
\thlabel{categorification}
Let $\Delta$ be the simply laced Dynkin diagram which underlies the quiver $Q$. Let $F$ be the autoequivalence $\Sigma \circ \tau^{-1}:\cd_Q \to \cd_Q$ and $C\subset \Z Q$ be an admissible configuration invariant under $F$. 
\begin{itemize}
\item[$(i)$]
Then  $\gpr(\cs_C) \corb F_*$ is a Frobenius 2-Calabi-Yau realization $($in the sense of \cite{Fu Keller}$)$ of a cluster algebra with geometric coefficients of type $\Delta$.
\item[$(ii)$] If $C =\Z Q$ $($\ie $\cs_C = \cs)$ then  $\gpr(\cs)\corb F_*$ is a 2-Calabi-Yau realization of the cluster algebra with universal coefficients of type $\Delta$ $($\cf \cite{clusters 4}$)$. 
\end{itemize}
\end{theorem}

\section{Frobenius models of $\cd_Q/F$}
In this section we prove \thref{Frobenius model orbit}. 

\subsection{The stratifying functor}
The functor $\Phi:\mod(\cs_C)\to \cd_Q$ of \thref{description of gpr} was named in \cite{Keller Scherotzke 1} the \emph{stratifying functor}. It can be used to define a stratification of Nakajima's graded affine quiver varieties (introduced in \cite{Nakajima}) whose strata are parameterized by the objects of $\ind(\cd_Q)$. We will show that $\phi:\ul{\gpr}(\cs_C) \to \cd_Q$ is induced by a quasi-functor. 
The inverse of $\phi:\ul{\gpr}(\cs_C) \to \cd_Q$ can be described explicitly as the composition of two functors as follows:
on the one hand we consider the path category $kQ$ as a full subcategory of $\cR_C$ via the embedding $i \mapsto (i,0)$. The restriction functor gives a functor $kQ\to \cs_C$ taking $x$ to $\res(x^{\wedge})$. It gives rise to a $kQ$-$\cs_C$-bimodule $X$ given by
\begin{equation*}
X(u,x)= \Hom (u^{\wedge},\res(x^{\wedge})), \ \text{for }\ x\in Q_0\ \text{and} \ u\in\sigma(C)
\end{equation*}
and therefore a functor
\begin{equation*}
?\lten_{kQ} X : \cd_Q \to \cd^b(\mod(\cs_C)).
\end{equation*}
Notice that for every $kQ$-module $M$ the $\cs_C$-module $M\otimes_{kQ} X$ lies in $\gpr(\cs_C)$. By definition, the derived category of $\gpr(\cs_C)$ (as an exact category) can be identified with a full triangulated subcategory of $\cd^b(\mod(\cs_C))$. Therefore, we can consider the derived tensor product as a triangulated functor
\begin{equation*}
?\lten_{kQ} X : \cd_Q \to \cd^b(\gpr(\cs_C)).
\end{equation*}
On the other hand, since $\gpr(\cs_C)$ is a Frobenius category, there is a canonical triangulated functor 
\begin{equation*}
\can : \cd^b(\gpr(\cs_C)) \to \ul{\gpr}(\cs_C).
\end{equation*}
constructed as follows. Let $\cp $ be the full subcategory of $\gpr(\cs_C)$ formed by its projective objects and denote by $\ch^-(\cp)$ the homotopy category associated to the category of bounded above complexes with components in $\cp$. There is a functor
\begin{equation*}
\label{Rickard equivalence}
{\bf p}:\cd^b(\gpr(\cs_C))\to \ch^-(\cp) 
\end{equation*}
sending a complex $X$ to a quasi-isomorphic complex ${\bf p}X\in  \ch^-(\cp)$. If $P^{\bullet}$ is a complex in $\ch^-(\cp)$ and $p\in\Z$ is small enough, then the objects $\Sigma^{-p}(Z^p(P^{\bullet}))\in \ul{\gpr}(\cs_C)$ are canonically isomorphic. Moreover, if $P_1^{\bullet}$ and $P^{\bullet}_2$ are quasi-isomorphic complexes in $\ch^-(\cp)$ and $p\ll 0$ then we have an isomorphism $\Sigma^{-p}(Z^p(P^{\bullet}_1))\cong \Sigma^{-p}(Z^p(P^{\bullet}_2))$ in $\ul{\gpr}(\cs_C)$. Thus there is a functor
\begin{equation*}
t: \ch^-(\cp)\to \ul{\gpr}(\cs_C)
\end{equation*}
sending a complex $P^{\bullet}$ to $t(P^{\bullet}):= \Sigma^{-p}(Z^p(P^{\bullet}))$ for some $p \ll 0$ (which depends on $P^{\bullet}$). 
Then $\can:\cd^b(\gpr(\cs_C)) \to \ul{\gpr}(\cs_C)$ is defined as the composition of the functors described above
\begin{equation*}
\can=t\circ {\bf p}.
\end{equation*}
Finally, the functor $\phi^{-1}:\cd_Q \to \ul{\gpr}(\cs_C)$
is given by the composition
\[
\xymatrix@C=1.2cm{
\phi^{-1}:\cd_Q \ar[r]^-{?\lten_{kQ} X} &  \cd^b(\gpr(\cs_C)) \ar[r]^-{\can} &  \ul{\gpr}(\cs_C).
}
\]
\begin{remark}
The fact that $\phi^{-1}$ is a triangle equivalence is proved in Section 5 of \cite{Keller Scherotzke 1}.
\end{remark}

\subsection{A dg lift of $\phi^{-1}$}
In this subsection we prove a technical result which we shall need to compare the triangulated structures on the orbit categories that we consider. Namely, we prove that the functor $\phi^{-1}$ admits a dg lift in the sense of \cite{Keller triang}. In other words, we prove that $\phi^{-1}$ is the triangulated functor associated to a quasi-functor between pretriangulated dg categories.

\begin{notation}
Let $C\subset \Z Q_0$ be an admissible configuration. For the rest of the paper we let $\ce_C$ be the Frobenius category $\gpr (\cs_C)$. In the rest of this section we will drop the subindex $C$ and for simplicity just write $\ce$. We let $\cp$ be the full subcategory of $\ce$ formed by its projective objects.
\end{notation}

Let $ C^b(\proj(kQ))_{dg}$ be the dg category of bounded complexes of projective $kQ$-modules and $\ac (\cp)_{dg}$ be the dg category of acyclic complexes of objects of $\cp$. The dg categories $ C^b(\proj(kQ))_{dg}$ and  $\ac (\cp)_{dg}$ are dg enhancements of the categories $\cd_Q$ and $\ul{\ce}$, respectively. The dg quotient
\begin{equation*}
\cd^b(\ce)_{dg}:= C^b(\ce)_{dg}/ \ac^b(\ce)_{dg}
\end{equation*}
is a dg enhancement of $\cd^b(\ce)$. We will show that the equivalence $\phi^{-1}$ is the triangulated functor associated to a quasi-functor $ C^b(\proj(kQ))_{dg}  \to \ac(\cp)_{dg}$. To prove that the functor $\can: \cd^b(\ce)\to \ul{\ce}$ is algebraic we need the following.

\begin{definition}
Let $\text{dgcat}_k$ denote the category of small dg categories over $k$. It admits the structure of a model category whose weak equivalences are the quasi-equivalences (\cf \cite{Tabuada model}). Let $\text{Hqe}$ denote the associated homotopy category. Let $\cc$ be a small dg category and $\cb$ a full dg subcategory of $\cc$. We say that a morphism $G: \cc \to \cc'$ in $\text{Hqe}$ \emph{annihilates} $\cb$ if its associated functor $H^0(G):H^0(\cc) \to H^0(\cc')$ takes all objects of $\cb$ to zero objects.
\end{definition}

\begin{theorem}$($\cite{Keller on dg, Tabuada quotient}$)$
\thlabel{Tabuada's theorem}
There is a morphism $ \cc \to \cc / \cb $ of $\text{Hqe}$ which annihilates
$\cb$ and is universal among the morphisms annihilating $\cb$.
Moreover, the morphism $\cc \to \cc / \cb $ induces an equivalence 
between the category of quasi-functors $\cc/\cb \to \cc'$ into the category of quasi-functors $\cc \to \cc'$ whose associated functor annihilates $\cb$.
\end{theorem}

\begin{proposition}
\thlabel{lift of phi}
There is a quasi-functor
\begin{equation*}
\tilde{\phi}^{-1}:  C^b(\proj(kQ))_{dg} \to \ac(\cp)_{dg} 
\end{equation*}
such that $H^0(\tilde{\phi}^{-1})\cong\phi^{-1}$.
\end{proposition}
\begin{proof}
Recall that $\phi^{-1}$ is defined as the composition
\[
\xymatrix@C=1.2cm{
\phi^{-1}:\cd_Q \ar[r]^-{?\lten_{kQ} X} &  \cd^b(\ce) \ar[r]^-{\can} &  \ul{\ce}.
}
\]
It is enough to prove that these two functors admit a dg lift. Let $P^{\bullet}X$ be a projective resolution of $X$ as a $kQ$-$\cs_C$-bimodule. We can compose the dg functor
\begin{equation*}
-{?\otimes_{kQ} P^{\bullet} X}: C^b(\proj(kQ))_{dg} \to C^b(\ce)_{dg} 
\end{equation*}
with the canonical dg functor $C^b(\ce)_{dg} \to \cd^b(\ce)_{dg}$ to obtain a dg lift of $?\lten_{kQ} X : \cd_Q \to \cd^b(\ce)$. Recall that every module $M\in \ce$ is of the form $M\cong Z^0(P_M)$, for some complex $P_M\in \ac (\cp)$. Moreover, since projective resolutions and injective coresolutions of objects in $\ce$ can be chosen functorially, there is a faithful functor 
\begin{equation*}
i:\ce \rightarrow \ac(\cp).
\end{equation*} 
Therefore, we can consider $C^b(\ce)$ as a subcategory of the category ${\rm Bi}(\cp)$ of \emph{double complexes} with components in $\cp$ (see for instances {\bf Sign Trick 1.2.5} of \cite{Weibel}). Let ${\rm Tot}(B)$ be the \emph{completed }total complex associated to a double complex $B\in {\rm Bi}(\cp)$. By construction, we have that
\begin{equation*}
\can (M)\cong H^0({\rm Tot}(i(M))).
\end{equation*}
The composition ${\rm Tot}\circ i: C^b(\ce)\to \ac(\cp)$ defines a dg functor ${\rm Tot}\circ i:C^b(\ce)_{dg}\to \ac(\cp)_{dg}$ which annihilates $\ac^b(\ce)_{dg}$. In light of \thref{Tabuada's theorem}, there is a quasi-functor
\begin{equation*}
\widetilde{{\can}}:\cd^b(\ce)_{dg}\to \ac(\cp)_{dg}
\end{equation*}
whose associated triangle functor is ${\can}:\cd^b(\ce)\to \ul{\ce}$. Finally, we can define $\tilde{\phi}^{-1}$ as the quasi-functor associated to the composition of dg functors
\[
\xymatrix@C=1.2cm{
C^b(\proj(kQ))_{dg} \ar[r]^-{?\otimes_{kQ} P^{\cdot} X} &  C^b(\ce)_{dg} \ar[r]^-{\widetilde{\can}} & \ac(\cp)_{dg}.
}
\]
\end{proof}

\subsection{Proof of \thref{Frobenius model orbit}}

Recall that we are given an exact autoequivalence $F_*:\ce \iso \ce$ extending a triangle functor $F: \cd_Q\to \cd_Q$ such that $\cd_Q/F$ is $\Hom$-finite and triangulated with respect to $\widetilde{F}$ (see \eqref{canonical dg lift of F}). Since $F_*$ is exact, it restricts to an equivalence $F_*:\cp \to \cp$ and therefore it induces a triangulated functor $\ul{F_*}:\ul{\ce} \to \ul{\ce}$. Let $\widetilde{F}_*:\ac(\cp)_{dg}\to \ac(\cp)_{dg}$ be the dg functor defined as $F_*$ componentwise. Note that $\widetilde{F}_*$ it induces $\ul{F_*}$ in homology, \ie $H^0(\widetilde{F}_*)=\ul{F_*}$.

\begin{lemma}
\thlabel{a quasi-isomorphism}
Under the above assumptions, the dg category $\ac(\cp)_{dg}/\widetilde{F}_*$ is quasi-equivalent to $C^b(\proj kQ)_{dg}/\widetilde{F}$. In particular, the triangulated hull of $\ul{\ce}/\ul{F_*}$ is triangle equivalent to the triangulated hull of $\cd_Q/F$.
\end{lemma}
\begin{proof}
Let $\tilde{\phi}$ be an inverse of the quasi-functor $\tilde{\phi}^{-1}$. In view of the universal property of the triangulated hull (see Section 9.4 of \cite{Keller triang}) it is enough to show that the diagram

\begin{equation*}
\xymatrix{
\ac(\cp)_{dg} \ar[r]^-{M_{\tilde{\phi}}}\ar[d]_{M_{\widetilde{F}_*}} & C^b(\proj kQ)_{dg}  \ar[d]^{M_{\widetilde{F}}}\\
\ac(\cp)_{dg} \ar[r]^-{M_{\tilde{\phi}}} &C^b(\proj kQ)_{dg}.
}
\end{equation*}
is commutative up to isomorphism, where $M_X$ is the quasi functor associated to a dg functor $X$. By construction we have that $\widetilde{F} \phi \cong \phi \widetilde{F}_*$. Thus there is a isomorphism of quasi functors
\begin{equation*}
M_{\tilde{F}} M_{\tilde{\phi}} \cong M_{\tilde{\phi}} M_{\tilde{F}_*}.
\end{equation*}
\end{proof}

\begin{proof}[Proof of \thref{Frobenius model orbit}]
$(i)$ By \thref{a quasi-isomorphism} the category $\ul{\ce}/\ul{F_*}$ is triangulated and triangle equivalent to $\cd_Q/F$. By  Theorem 42 of \cite{Najera Frobenius orbit}, $\eafast$ admits the structure of a Frobenius category whose stable category is triangle equivalent to $\cd_Q / F$.

$(ii)$ It is clear that $\eafast$ is $k$-linear. By Theorem 42 of \cite{Najera Frobenius orbit}, we know that it is Krull-Schmidt (see also Lemma 35 of loc. cit.). It is $\Ext^1$-finite since $\cd_Q/F$ is $\Hom$-finite. So $\eafast$ satisfies condition (P0). 

We know that the canonical projection $p:\cd_Q \to \cd_Q/F$ is exact. 
Then by the proof of Proposition 1.3 of \cite{BMRRT} we have that $\cd_Q/F$ satisfies condition (P1) and that its AR quiver is $\Z Q / F$. Therefore, $\ce \corb F_*$ satisfies (P1) as well. 

Let $p(z^{\wedge})$ be an indecomposable  projective object of $\ce \corb F_*$ (See \thref{indecomposables of gpr} and \thref{indecomposables orbit}). So $z$ is some frozen vertex of $\Z Q_C$. 
Let $f$ be the morphism in $\ce$ corresponding to the arrow $z \to \sigma^{-1}(z)$ in $\Z Q_C$ under the Yoneda embedding $\cR_C \to \ce$. 
By \thref{Frobenius models}, $f$ is an irreducible morphism. We claim that the image of $f$ under the canonical projection $p: \ce \to \eafast$ is an irreducible morphism. Let $p(x^{\wedge}) \in \eafast $ be an indecomposable object which is not isomorphic to $p(z^{\wedge})$. Let $h: p(z^{\wedge}) \to p(x^{\wedge})$ be a non-zero morphism in $\eafast$.
Then $h=(h_i)$, with $h_i : z^{\wedge} \to F_{*}^i(x^{\wedge})$. Since $p(z^{\wedge})$ is not isomorphic to $p(x^{\wedge})$, we have that $F^i(x)\neq z$ for all $i\in \Z$. Since $f$ is irreducible, we have that $h_i= g_i \circ f$ for some $g_i: \sigma ^{-1}(z)^{\wedge} \to F_*^i(x^{\wedge}) $. Then $h= (g_i)\circ p(f)$. This shows that $p(f)$ is left almost split
and that $p(f)$ is the only irreducible morphism with source $p(z^{\wedge})$. Similarly, let $g$ be the irreducible morphism in $\ce$ corresponding to the arrow $\sigma(z)\to z$ in $\Z Q_C$. As before we conclude that $p(g)$ is the only irreducible morphism of $\eafast$ ending on $\sigma(z)^{\wedge}$. This shows that the map $p(x^{\wedge})\to x$ induces an isomorphism between the AR quiver of $\eafast$ and the quiver $\Z Q_C/F$. Moreover, the frozen vertices of $\Z Q_C$ correspond to the projective-injective objects of $\eafast$. 

To complete the proof of part $(ii)$ it only remains to show that $\eafast$ is standard. We have to check that there are no relations other that those induced by the mesh relations associated to non-projective indecomposable objects. Let $r\in \eafast(X,Y)$ be a relation in $\eafast$. So $r$ is a zero morphism that can be expressed as a finite sum non-zero morphism, that is
\[
r=\sum_{i=1}^k f_i
\]
where $f_i =(f_{ij})_{j\in \Z}\in \eafast (X,Y)$ is a non-zero morphism for each $i$ and $f_{ij}\in \ce (X,F^j(Y))$. In particular, we have that $r_{j}:=f_{1j} + \cdots f_{kj}=0$ for all $j \in \Z$, \ie $f_{1j} + \cdots f_{kj}$ is a relation in $\ce$. Since $\ce$ is standard we have that $r_{ij}$ is induced by the non-frozen mesh relations of its AR quiver. Since $p:\ce \to \eafast$ is exact, it follows that the same holds for $\eafast$.

$(iii)$ It only remains to check that $C$ is admissible. Let $\pi(x^{\wedge}) \overset{(f_i)}{\longrightarrow} \pi(I)$ be an inflation, where $I$ is an injective object of $\ce$ and $f_i:x^{\wedge}\to F^i(I)$ is a morphism of $\ce$. In particular, there is a path $p$ from $x$ to $\sigma^{-1}(c)$ for some $c\in C$. To finish the proof we can proceed exactly as in the proof of \thref{Frobenius models} given \cite{Keller Scherotzke 1}. 
\end{proof}

\section{Categorification}

In this section we prove \thref{categorification}. First, we recall the notion of cluster algebra with universal coefficients. We assume that the reader has some familiarity with cluster algebras, for general background on this theory we refer to the articles \cite{clusters 1,clusters 4}.

\subsection{Universal coefficients for finite-type quivers} Let $n$ be a positive integer, $[1,n]:=\lbrace 1, \dots , n \rbrace$ and $\T^n$ be the $n$-regular tree. To construct a rank $n$ cluster algebra with coefficients on a semifield $\P$ one needs to specify the following initial information:
\begin{itemize}
\item a quiver $Q$ without cycles of length $1$ or $2$ such that $Q_0=[1,n]:= \lbrace 1, \dots, n\rbrace$; 
\item an $n$-tuple of coefficients $\mathbf{y}=(y_1,\dots y_n)\in \P^n$.
\end{itemize}
We always assume that the initial cluster is $(x_1,\dots, x_n)$. The cluster algebra associated to this data will be denoted by $\ca_{\P}(Q,\mathbf{y})$ and is a subalgebra of $\Z \P[x_1^{\pm 1}, \dots,x_n^{\pm 1} ]$. The set of cluster variables of $\ca_{\P}(Q,\mathbf{y})$ will be denoted by $\lbrace x_{i,t}\rbrace$ where $i \in [1,n]$ and $t \in \T$. Similarly, its set of coefficients will be denoted by $\lbrace y_{i,t}\rbrace$.

\begin{definition}
An \emph{ice quiver} consist of a pair $(Q,f)$, where $Q$ is a quiver together with a distinguished subset $f \subset Q_0$ whose elements are called frozen vertices such that there are no arrows between any two vertices of $f$. When we refer to ice quivers we will assume that $Q_0\setminus f = [1,n]$ and $f=[n+1,n+m]$
\end{definition}

\begin{remark}
Cluster algebras with coefficients in a tropical semifield are called geometric cluster algebra. Note that a frozen quiver determines the initial information of a geometric cluster algebra. 
Therefore, we will systematically define geometric cluster algebras using ice quivers.
\end{remark}

\begin{definition}
Let $\ca= \ca_{\P}(Q,\mathbf{y})$ and $\ol{\ca}=\ol{\ca}_{\P'}(Q,\mathbf{y'})$ be cluster algebras with sets of cluster variables $(x_{i,t})$ and $(\ol{x}_{i,t})$, respectively. We say that $\ol{\ca}$ is obtained from $\ca$ by a \emph{coefficient specialization} if there is a (unique) homomorphism of multiplicative groups $\varphi: \P \to \P'$ that extends to a (unique) ring homomorphism $\varphi: \ca \to \ol{\ca}$ such that $\varphi(x_{i,t})=\ol{x}_{i,t}$ for all $i$ and all $t$. We call $\varphi$ a coefficient specialization.
\end{definition}

\begin{remark}
By Proposition 12.2 of \cite{clusters 4} we know that $\varphi$ is a coefficient specialization if and only if $\varphi(y_{i,t})=\ol{y_{i,t}}$ and $\varphi(y_{i,t}\oplus 1)=\ol{y_{i,t}}\oplus 1$ for all $y_{i,t}$. Here $t \mapsto (\mathbf{y_t},B_t)$ (resp. $t \mapsto (\mathbf{\ol{y}_t},B_t)$) is the underlying $Y$-pattern for $\ca$ (resp. $\ol{\ca}$).
\end{remark}
\begin{definition}
We say that a cluster algebra $\ca_{\P}(Q,\mathbf{y})$ has \emph{universal coefficients} if every cluster algebra of the form $\ol{\ca}_{\P'}(Q,\mathbf{y'})$ is obtained from $\ca_{\P}(Q,\mathbf{y})$ by a (unique) coefficient specialization.
\end{definition}
The existence of a cluster algebra with universal coefficients is not clear. If it exists, then it can be regarded as an invariant of the mutation class of a quiver (these cluster algebras do not depend on the choice of initial seed). In \cite{clusters 4}, the authors constructed cluster algebras with universal coefficnets for any \emph{cluster-finite quiver} (\ie a quiver mutation equivalent to a Dynkin quiver). In a series of papers \cite{Reading universal, Reading surfaces, Reading Markov}, Reading studied the existence of universal coefficients for cluster algebras with geometric coefficients beyond finite-type. For the convenience of the reader, we recall the construction of finite-type cluster algebras with universal coefficient worked out in \cite{clusters 4}.

\begin{definition}
Let $\Delta$ be a simply-laced Dynkin and $\Phi^{\Delta}$ its root system. We fix a set of simple roots $\lbrace \alpha_1, \dots , \alpha_n \rbrace$ and denote by $\Phi_{\geq -1}$ the set of \emph{almost positive roots}, \ie the union of the set of positive roots $\Phi_{>0}$ and the set of negative simple roots. If $\alpha\in \Phi^{\Delta}$, then $[\alpha:\alpha_i]$ denotes the multiplicity of $\alpha_i$ in $\alpha$. Let $\P_{\Delta}^{\rm univ}:=\Trop(p_{\alpha}:\alpha \in \Phi_{\geq -1})$ be the tropical semifield whose generators $p_{\alpha}$ are labeled by the set $\Phi_{\geq -1}$. 
\end{definition}

Recall that every bipartite orientation $Q$ of $\Delta$ is endowed with a function $\varepsilon: Q_0\to \lbrace +, - \rbrace$ that associates a \emph{sign} to each vertex:
\begin{equation*}
\varepsilon (i)=
\begin{cases}
+ & \text{if } i \ \text{is a source},\\
- & \text{if } i \ \text{is a sink}.
\end{cases}
\end{equation*}

\begin{theorem} $($\cite[Theorem 12.4]{clusters 4}$)$
\thlabel{th universal coefficients}
Let $Q$ be a bipartite orientation of a simply laced Dynkin diagram $\Delta$. Let $\ca_{Q}^{\rm univ}$ be the cluster algebra with coefficients defined over $\P_{\Delta}^{\rm univ}$ whose initial quiver is $Q$ and whose $n$-tuple of initial coefficients is $(y_1,\dots,y_n)$ where
\begin{equation}
\label{initial universal coefficients}
y_j = \prod_{\alpha \in \Phi_{\geq -1}}
p_{\alpha}^{\varepsilon(j) [\alpha : \alpha_j]} .
\end{equation}
Then $\ca_{Q}^{\text{univ}}$ has universal coefficients.
\end{theorem}

\begin{remark}
\thlabel{a remark on frozen quivers}
Since $\ca_{Q}^{\rm univ}$ turns out to be a geometric cluster algebra it can be completely described by an ice quiver. The equation \eqref{initial universal coefficients} allows us to describe at once an ice quiver determining $\ca_{Q}^{\rm univ}$. 
\end{remark}

\subsection{Frobenius 2-CY realizations}
Throughout this subsection we will assume that all categories are Krull-Scmidt. We refer the reader to \cite{Krause Krull-Schmidt} for background on Krull-Schmidt categories.
Let $\ce$ be a Frobenius category. We say that $\ce$ is \emph{stably 2-Calabi-Yau} (stably 2-CY for short) if its stable category $\ul{\ce}$ is $\Hom$-finite and 2-CY as a triangulated category. This means that for every pair of objects $X$ and $Y$ of $\ul{\ce}$ there is a bifunctorial isomorphism 
\begin{equation*}
\Hom_{\ul{\ce}}(X,Y) \cong D \Hom_{\ul{\ce}}(Y, \Sigma^2 (X)),
\end{equation*}
where $\Sigma$ is the suspension functor of $\ul{\ce}$. Note that every stably 2-CY Frobenius category is $\Ext^1$-finite (\ie for every pair of objects $X$ and $Y$ of $\ce$ the dimension of $\Ext_{\ce}^1(X,Y)$ is finite). 

\begin{definition}
Suppose that $\cc$ is either an $\Ext^1$-finite Frobenius category or a $\Hom$-finite triangulated category. A {\em cluster-tilting subcategory of $\cc$} is a full additive subcategory $\ct\subset \cc$ which is stable under taking direct factors and such that
\begin{itemize}
\item[$(i)$] for each object $X$ of $\cc$ the functors $\cc(X,?)|_{\ct}:\ct\to \Mod (k)$ and $\cc(?,X):\ct\op\to\Mod (k)$ are finitely generated; 
\item[$(ii)$] an object $X$ of $\ce$ belongs to $\ct$ if and only if we have $\Ext^1_{\cc}(T,X)=0$ for all objects $T$ of $\ct$.
\end{itemize}
An object $T$ is called {\em cluster-tilting} if it is basic (its indecomposable summands are pairwise distinct) and $\add(T)$ is a cluster-tilting subcategory. Equivalently, $T$ is cluster-tilting if and only if it is rigid and each object $X$ satisfying $\Ext^1_{\ce}(T,X)=0$ belongs to $\add(T)$. 
\end{definition}

\begin{remark}
We will identify a basic cluster-tilting object $T$ with the cluster-tilting subcategory $\add(T)$ and occasionally refer to $T$ as a cluster-tilting  subcategory.
\end{remark}

The following result is straightforward and its proof is left to the reader.

\begin{lemma}
\thlabel{on ct subcategories}
Let $\ce$ be an $\Ext^1$-finite Frobenius category. Then the natural projection $\ce\to \ul{\ce}$ induces a bijection between the cluster-tilting subcategories of $\ce$ and the cluster-tilting subcategories of $\ul{\ce}$.
\end{lemma}

The quiver of an additive subcategory $\ct \subset \cc$ is denoted by $Q_{\ct}$. By definition, the set of vertices of $Q_{\ct}$ is the set of isomorphism classes of indecomposable objects in $\ct$. The arrows from the class of an object $X$ to the class of an object $Y$ corresponds to a basis of the space of irreducible maps $\rad(X,Y)/\rad^2(X,Y)$, where $\rad (\ ,\ )$ denotes the radical in $\ct$, \ie the ideal of all non-isomorphisms from $X$ to $Y$. The quiver of a subcategory $\ct$ of $\cc$ can be though of as an ice quiver by considering the vertices corresponding to the indecomposable projective objects of $\cc$ that lie in $\ct$ to be frozen and neglect the arrows between them. We will write $Q^{fr}_{\ct}$ for the resulting ice quiver. We say that $\cc$ has no loops (resp. 2-cycles) if the ice quiver of every cluster-tilting subcategory has no loops (resp. 2-cylces).

\begin{remark}
\thlabel{semiperfect rings}
Let $T$ be a basic object of $\cc$ and let $\ct=\add (T)$. Consider the basic algebra $A=\text{End}_{\cc}(T)$. If $\cc$ is Hom-infinite then in general $A$ will be infinite-dimensional. So it is not immediate that $A$ can be expressed as the quotient of a path algebra by an admissible ideal. Nevertheless, the fact that $\cc$ is Krull-Schmidt implies that $A$ is a semiperfect ring \cite[Corollary 4.4]{Krause Krull-Schmidt}. By definition, this means that every finitely generated right $A$-module has a projective cover. An equivalent condition is that every simple right $A$-module is of the form $eA / e\rad(A)$, where $e\in A$ is an idempotent (for more background on semi-perfect rings see \cite{AF}). Therefore, we can associate a quiver to $A$ computing the dimension of $\Ext^1(S,S')$ for every parir of simple $A$-modules $S$ and $S'$. Then the quiver of $A$ will be equal to $Q_{\ct}$.
\end{remark}

\begin{definition}[\cite{BIRS}]
\label{def:ExactClusterStructure}
Let $\ce$ be an $\Ext^1$-finite Frobenius category whose stable category is 2-CY. The cluster-tilting subcategories of $\ce$ {\em determine a cluster structure on $\ce$} if the following conditions hold:
\begin{itemize}
\item[0)] There is at least one cluster-tilting subcategory in $\ce$.
\item[1)] For each cluster-tilting subcategory $\ct$ of $\ce$ and each non-projective indecomposable object $X$ of $\ct$, there is a non-projective indecomposable object $X^*$ of $\ce$, unique up to isomorphism and not isomorphic to $X$, such that the additive subcategory $\ct'=\mu_X(\ct)$ of $\ce$ whose set of indecomposable objects
\[
\ind(\ct')=\ind(\ct)\setminus\{X\} \cup \{X^*\} 
\]
is a cluster-tilting subcategory.
\item[2)] In the situation of 1), there are conflations
\begin{center}
\begin{tabular}{lll}
$0\longrightarrow X^* \overset{g}\longrightarrow E \overset{h}\longrightarrow X \longrightarrow 0$
 &  and 
& $0\longrightarrow X \overset{s}\longrightarrow E' \overset{t}\longrightarrow X^*  \longrightarrow 0$,
\end{tabular}
\end{center}
where $h$ and $t$ are minimal right $\add (\ct \setminus  \lbrace X \rbrace)$-approximations
and $g$ and $s$ are minimal left $\add (\ct \setminus  \lbrace X \rbrace)$-approximations. We call these 
sequences the \emph{exchange conflations associated} to $\ct$ and $\ct'$.
\item[3)] $\ce$ has no loops or 2-cycles.
\item[4)] For each cluster-tilting subcategory $\ct$ of $\ce$ and each non-projective indecomposable object $X$ of $\ct$ we have that $Q^{fr}_{\mu_X(\ct)}$ and $\mu_{X}(Q^{fr}_{\ct})$ are related by quiver mutation, \ie $Q^{fr}_{\mu_X(\ct)}=\mu_{X}(Q^{fr}_{\ct})$. We say that $\ce$ (resp.)
 \end{itemize}
 \end{definition}

\begin{theorem}$($\cite[Theorem II.1.6]{BIRS}$)$
Let $\ce$ be a stably 2-CY Frobenius category with some cluster-tilting subcategory. If $\ce$ has no loops or 2-cycles, then the cluster-tilting subcategories determine a cluster structure for $\ce$.
 \end{theorem}

\begin{definition} \cite{Fu Keller}
A Frobenius category $\ce$ is a \emph{$2$-Calabi-Yau realization} of the cluster algebra associated to an ice quiver $(Q,f)$ if $\ce$ has a cluster structure and has a cluster-tilting object $T$, such that the ice quiver of $\add (T)$ is $(Q,f)$.
\end{definition}

\subsection{The quiver of a cluster-tilting subcategory.}

We return to the context of Nakajima categories and assume that $Q$ is a Dynkin quiver. For the rest of the paper $F$ will denote the autoequivalence $\Sigma \circ \tau^{-1}:\cd_Q\to \cd_Q$. Moreover, we let $C \subset \Z Q$ be an $F$-invariant configuration, and denote by $\ce$ the Frobenius category $\gpr(\cs_C) $ and by $F_{\ast}:\ce \to \ce$ the exact functor extending $F$.

\begin{lemma}
\thlabel{ct subcategories}
The canonical projection $p:\cd_Q\to \cc_Q$ induces a bijection between the cluster-tilting subcategories of $\cd_Q$ and the cluster-tilting subcategories of $\cc_Q$.
\end{lemma}
\begin{proof}
This is a direct consequence of \cite[Proposition 2.2]{BMRRT} where cluster-tilting subcategories are called $\Ext$-configurations. Explicitly, if $T$ is a basic cluster-tilting object of $\cc_Q$, then under this bijection the cluster-tilting subcategory $\add(T)$ corresponds to the cluster-tilting subcategory $\ct=\add ( F^j(T) \ |\  j\in \Z )$. Moreover, $\add(T)$ is canonically isomorphic to $\ct /F$.
\end{proof}

\begin{remark}
\thlabel{cts}
\thref{on ct subcategories} and \thref{ct subcategories} together tell us that  the cluster-tilting subcategories of  $\cd_Q,\ \cc_Q,\ \ce$ and $\ce \corb F_*$ are in bijection. Moreover, any cluster-tilting subcategory of $\ce \corb F_*$ is of the form $\ctt \corb F_*$ for some cluster-tilting subcategory $\ctt$ of $\ce$. From now on we fix such a subcategory $\ctt$ and denote its stable category by $\ct$. The diagram of Figure \ref{fig:bijections of ct subcategories} is helpful to keep track of the notation we will use. We let $\ct / F= \add(T)$, where $T=T_1\oplus \cdots\oplus T_n$ is a basic cluster-tilting object of $ \cc_Q$ and the $T_i$'s are its indecomposable summands. Moreover, if $P_1, \cdots, P_r$ are the indecomposable projective objects of $\ce\corb F_*$ and $\widetilde{T}=T\oplus P_1\oplus \cdots \oplus P_r$, then $\ctt \corb F_*$ is equivalent to $\add(\widetilde{T})$. We would like to compute the quiver of $\ctt \corb F_*$. In view of \thref{semiperfect rings} it is enough to compute an explicit (co)resolution of the simple $\ctt \corb F_*$-modules. Moreover, these simple modules are one-dimensional, supported on a single indecomposable object of $\ctt\corb F_*$.

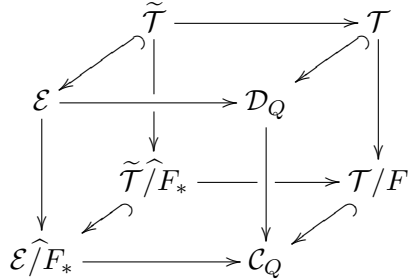
\begin{figure}[htbp]
\[
\xymatrix@!0 @R=2.5pc @C=3.5pc{
 & \ctt  \ar[rr] \ar@{^{(}->}[dl] \ar[dd]|\hole & &  \ct 
\ar[dd] \ar@{^{(}->}[dl] \\
 \ce  \ar[dd] \ar[rr]  & &
\cd_Q  \ar[dd] & \\
& \ctt \corb F_{\ast} \ar[rr]|\hole \ar@{^{(}->}[dl] &    & \ct/F
\ar@{^{(}->}[dl] \\
 \eafast \ar[rr] &   & \cc_Q &
}
\]
\caption{Bijection between cluster-tilting subcategories.}
\label{fig:bijections of ct subcategories}
\end{figure}
\end{remark}

We consider the quotient category $\cd_Q/\ct$ and let 
\[
x_{\cd / \ct}^{\wedge}= \cd_Q/\ct (?,x) 
\]
be the $\cd_Q/\ct$-module associated to an object $x$ of $\cd_Q$. Note that if $x$ lies in $\ct$ then $x^{\wedge}_{\cd/\ct}$ is the zero module. The module $x_{\cd/\ct}^{\vee}$ is defined analogously. Both $x_{\cd/\ct}^{\wedge}$ and $x_{\cd/\ct}^{\vee}$ are in a natural way $\cR_C$-modules. Consider the finitely generated projective $\cR_C$-module 

\[
P_{x,\ct}:= \bigoplus_{\sigma(y)\in C} \cd_Q/\ct (y,x) \otimes \sigma(y)^{\wedge}.
\]
The multiplicity of $\sigma(y)^{\wedge}$ in $P_{x,\ct}$ equals the dimension of the vector space $\cd_Q/\ct (y,x)$. In a completely analogous way, we define the finitely cogenerated injective $\cR_C$-module
\[
I_{x,\ct}:= \prod_{\sigma^{-1}(y)\in C} D\cd_Q/\ct (x,y) \otimes \sigma^{-1}(y)^{\vee}.
\]
The modules of the form $x^{\wedge}_{\cd/\ct}$ and $x^{\vee}_{\cd/\ct}$ will be very helpful to compute the resolutions of the simple $\ctt$-modules. Note that in practice these modules can be computed quite explicitly. 

\begin{definition} (\cite{Keller Reiten})
Let $T=\bigoplus_{i=1}^{n}T_i$ be the decomposition of $T$ as a direct sum of indecomposable objects. For any object $X$ of $\cc_Q$ there is a triangle 
\begin{equation}
\label{index triangles}
T_1^{X}\to T_0^{X}\to X \to \Sigma T^X_1
\end{equation}
where $T^X_0=\bigoplus_{i=1}^{n}T^{\oplus a_i}_i$ and $T^X_1=\bigoplus_{i=1}^{n}T^{\oplus b_i}_i$. With this notation, the \emph{index of} $X$ \emph{with respect to} $T$ is the integer vector 
\[
\text{ind}_T(X)=(a_1-b_1, \dots, a_n-b_n).
 \]
\end{definition}
\begin{remark} 
Even though the triangle in (\ref{index triangles}) is not unique, it can be proved that the index is well defined. Moreover, it was proved in \cite{Keller Reiten} that in any such triangle the morphism $T_0^{X}\to X$ is a left $\ct$-approximation. When this morphism is a minimal left $\ct$-approximation we will call the triangle minimal.
\end{remark}

\begin{notation}
$(i)$ Let $z=z_1\oplus \dots \oplus z_s$  be the decomposition of an object $z\in \cd_Q$ as a direct sum of indecomposable objects (see \thref{identification}). The projective $\cR_C$-module $z_1^{\wedge}\oplus \dots \oplus z_n^{\wedge}$ will be denoted by $z^{\wedge}$.  We proceed analogously for injective modules.

$(ii)$ The Yoneda functor induces an equivalence between $\cR_C$ and $\ind(\ce)$. So any $\cR_C$-module can be extended to an $\ce$-module. We denote the extension of $x^{\wedge}$ by $x^{\wedge}_{\ce}$. The restriction of $x^{\wedge}_{\ce} $ to $\ctt$ will be denoted by $x^{\wedge}_{\ctt}$. We proceed analogously for injective modules. Moreover, the $\cR_C$-modules of the form $P_{x,\ct}$ and $I_{x,\ct}$ can be first extended to $\ce$ and then restricted to $\ctt$. The finitely generated projective $\ctt$-modules obtained in this way will be denoted by $P_{x, \ctt}$ and $I_{x,\ctt}$, respectively.
\end{notation}

\begin{lemma}
Let $x\in \Z Q_0$ be a non-frozen vertex considered as an indecomposable object of $\cd_Q$ (see \thref{identification}). Let $ T_1^{x}\to T_0^{x}\to x \to \Sigma T^x_1$ be a triangle in $\cd_Q$ such that $T_0^{x}$ and $ T_1^x$ are objects of $\ct$, and $T_0^x\to x$ is a minimal left $\ct$-approximation. Then there is a minimal projective resolution of $\cR_C$-modules
\begin{equation}
\label{resolution of x D/T}
0\to {T_1^x}^{\wedge} \to P_{x,\ct}\oplus {T_0^x}^{\wedge} \to x^{\wedge}\to x_{\cd/\ct}^{\wedge}\to 0
\end{equation}
and a minimal injective coresolution of $\cR_C$-modules 
\begin{equation}
0\to x_{\cd/\ct}^{\vee} \to x^{\vee} \to I_{x,\ct}\oplus {\Sigma T_0^x}^{\vee} \to  {\Sigma T_1^x}^{\vee} \to 0.
\end{equation}
\end{lemma}
\begin{proof}
We can lift the triangle $ T_1^{x}\to T_0^{x}\to x \to \Sigma T^x_1$ to a conflation in $0\to T_1^{x} \to P\oplus T_0^{x}\to x \to 0 $ in $\ce$, where $P$ is a projective object and $ P\oplus T_0^{x}\to x$ is a minimal left $\ctt$-approximation. Applying the left exact functor $(?)^{\wedge}:\cR_C \to \mod(\cR_C)$ to this conflation we obtain the exact sequence
\begin{equation*}
0\to {T_1^{x}}^{\wedge} \to P^{\wedge}\oplus {T_0^{x}}^{\wedge} \to x^{\wedge}\to M\to 0.
\end{equation*}
By definition, $M(y)$ corresponds to all morphisms $y\to x$ that cannot be factorized through the minimal left $\widetilde{\ct}$-approximation $ P\oplus T_0^{x} \to x$. That is, all morphism from $y\to x$ in $\ce/\ctt\cong \cd_Q/\ct$. This shows that $M$ is isomorphic to $x^{\wedge}_{\cd/\ct}$. Therefore, we have obtained a projective resolution of $x^{\wedge}_{\cd/\ct}$ which is minimal because $T^x_{0}\to x$ is a minimal left $\ct$-approximation. It only remains to describe $P^{\wedge}$ explicitly. We know that $P^{\wedge}$ is a direct sum of indecomposable projective $\cR_C$-modules. To calculate the multiplicity of $\sigma(y)^{\wedge}$ in $P^{\wedge}$ it is enough to calculate the dimension of $\Ext_{\cR_C}^1(x_{\cd/\ct}^{\wedge}, S_{\sigma(y)})$, where $S_{\sigma(y)}$ is the simple $\cR_C$-module supported in $\sigma(y)$.  We do this calculating the value of the derived functor 
\[
\text{R}\Hom(x_{\cd /\ct }^{\wedge}, S_{\sigma(y)})= \text{R}\Hom(x_{\cd /\ct }^{\wedge}, \sigma(y)^{\vee}\to y^{\vee})= (0\to D\cd/\ct(y,x)).
\]
We conclude that the multiplicity of $\sigma(y)^{\wedge}$ in $P$ is equal to the dimension of $\cd/\ct(y,x)$. Which implies that $P^{\wedge}\cong P_{x,\ct}$. To obtain the injective coresolution we proceed in an analogous way, we apply the right exact functor $( ? )^{\vee}=D(\cR_C(?, \ )): \cR_C \to \mod (\cR_C) $ to a short exact sequence $0\to x \to I \oplus \Sigma T^{x}_0 \to \Sigma T^{x}_1\to 0$ lifting the triangle $T_0^x\to x \to \Sigma T_1^x \to \Sigma T_0^x$. We obtain an exact sequence
\[
0\to x^{\vee}_{\cd / \ct} \to  x^{\vee}\to I^{\vee} \oplus {\Sigma T^{x}_0}^{\vee} \to {\Sigma T^{x}_1}^{\vee}\to 0.
\]
As before we can deduce that $I$ is isomorphic to $I_{x,\ct}$.

\end{proof}

\begin{corollary}
\thlabel{resolutions upstairs 1}
Let $x$ be a vertex in $C$ and $ T_1^{x}\to T_0^{x}\to x \to \Sigma T^x_1$ be a triangle in $\cd_Q$ such that $T_0^{x}$ and $ T_1^x$ are objects of $\ct$, and $T_0^x\to x$ is a minimal  $\ct$-approximation. Then the minimal projective resolution of the simple $\ctt$-module $S_{\sigma^{-1}(x)^{\wedge}}$ supported on $\sigma^{-1}(x)^{\wedge}$ is given by

\[
0 \to  {T_1^{x}}^{\wedge}_{\ctt} \to  P_{x, \ctt}^{\wedge}\oplus {T_0^{x}}_{\ctt}^{\wedge}\to {\sigma^{-1}(x)_{\ctt}^{\wedge}} \to S_{\sigma^{-1}(x)^{\wedge}} \to 0,
\]
and its minimal injective coresolution is given by

\begin{equation}
0\to S_{\sigma^{-1}(x)^{\wedge}}  \to x^{\vee}_{\ctt} \to I_{x,\ctt}\oplus {\Sigma T_0^x}^{\vee}_{\ctt} \to  {\Sigma T_1^x}^{\vee}_{\ctt} \to 0.
\end{equation}
\end{corollary}
\begin{proof}
From the description of the AR-quiver of $\ce$, we know that the simple $\ce$-module $S_{\sigma^{-1}(x)^{\wedge}}$ supported on the indecomposable object $\sigma^{-1}(x)^{\wedge}$ admits a projective resolution 
\[
0  \to x^{\wedge}_{\ce} \to \sigma^{-1}(x)^{\wedge}_{\ce} \to S_{\sigma^{-1}(x)^{\wedge}} \to 0
\]
Applying the restriction functor res$:\Mod(\ce)\to \Mod(\ctt)$ to this sequence we obtain the exact sequence  of $\widetilde{\ct}$-modules
\[
 0\to x^{\wedge}_{\ctt} \to \sigma(x)^{\wedge}_{\ctt} \to S_{{\sigma^{-1}(x)}^{\wedge}} \to 0.
\]
Note that the resolution of $\cR_C$ (\ref{resolution of x D/T}) induces an exact sequence of $\widetilde{\ct}$-modules
\[
0\to {T_1^x}_{\ctt}^{\wedge} \to P_{x,\ctt}\oplus {T_0^x}_{\ctt}^{\wedge}\to x^{\wedge}_{\ctt}\to 0.
\]
We can splice these last two sequences to obtain the minimal projective resolution. The minimal injective coresolution is obtained analogously.
\end{proof}

To obtain a (co)resolution of the $\ctt$-modules associated to non-frozen vertices we need to recall the notion of mutation of cluster-tilting objects in $\cc_Q$ (\cf \cite{Iyama Yoshino}). Let $T$ be a basic cluster-tilting object of $\cc_Q$ and $T_i$ one of its indecomposable summands. Then the exchange triangles associated to $T$ and $T_i$ are triangles of the form

\begin{center}
\begin{tabular}{lllll}
$\phantom{l}T_i^{\ast} \overset{\ul{g}}\longrightarrow B \overset{\ul{h}}\longrightarrow T_i \to \Sigma T_i^*$
&
&
and
&
&
$ T_i \overset{\ul{s}}\longrightarrow B' \overset{\ul{t}}\longrightarrow T^{\ast}_i \to \Sigma T_i$,
\end{tabular}
\end{center}
where $\ul{h}$ and $\ul{t}$ are minimal right $\add (\ctt \setminus  \lbrace T_i\rbrace)$-approximations
and $\ul{g}$ and $\ul{s}$ are minimal left $\add (\ctt \setminus  \lbrace T_i \rbrace)$-approximations. Changing the objects appearing in these triangles by isomorphic objects if necessary, we can assume that these sequences are induced by triangles in $\cd_Q$ of the form
\begin{align}
\phantom{xxxxxl} T_i^{\ast} \longrightarrow B \longrightarrow T_i \to \Sigma T_i^*
&
&
\phantom{l}\text{and}
&
&
 F^{-1}(T_i) \longrightarrow B' \longrightarrow T^{\ast}_i \to \Sigma F^{-1}(T_i),
\end{align}
respectively. We lift these triangles to conflations in $\ce$ of the form
\begin{center}
\begin{tabular}{lllll}
$\phantom{xxxll}  0 \to T_i^{\ast} \to E \to T_i \to 0$
&
&
and
&
&
$0\to F_*^{-1}T_i \to E' \to T^{\ast}_i\to 0$,
\end{tabular}
\end{center}
respectively.
\begin{lemma}
\thlabel{resolutions upstairs 2}
Let $x$ be a vertex of $\Z Q$ considered as an indecomposable object of $\cd_Q$ such that $p(x)=T_i$. With the notation above, the minimal projective resolution of the simple $\widetilde{\ct}$-module $S_{x^{\wedge}}$ supported on $x^{\wedge}$ is given by
\[
0 \to F^{-1}(x)^{\wedge}_{\ctt} \to E'^{\wedge}_{\ctt} \to E^{\wedge}_{\ctt} \to x^{\wedge}_{\ctt} \to S_{x} \to 0,
\]
the minimal injective coresolution of the simple $\widetilde{\ct}$-module $S_{x^{\wedge}}$ is given by
\[
0 \to S_x \to x^{\vee}_{\ctt}\to E'^{\vee}_{\ctt} \to E^{\vee}_{\ctt} \to F(x)^{\vee}_{\ctt} \to 0 
\]
\end{lemma}
\begin{proof}

Notice that the exchange conflations remain exact when applying the left exact functor $( \ )^{\wedge}: \widetilde{\ct} \to \mod \widetilde{\ct}$. We can splice the two resulting sequences to obtain the minimal projective resolution. The coresolution is obtained analogously.
\end{proof}
We have built the ground to describe the projective resolutions of the simple $\taf$-modules. By slight abuse of notation, we denote by $p$ the natural projection $\ce \to \ce \corb F_*$ (we used $p$ to denote the natural projection $\cd_Q\to \cc_Q$, but this last projection is induced by $\ce \to \ce \corb F_*$). Restricting $p$ to $\ctt$, we obtain a pair of adjoint functors
\[
\xymatrix{
\Mod (\ctt) \ar_{\pi}@<-0.5ex>[d]  \\
\Mod (\taf) \ar_{p^{\ast}}@<-0.5ex>[u] 
}
\]
where $p^{\ast}$ is the restriction functor and $\pi$ its left adjoint. At the level of objects $p_{\ast}$ is defined as $\pi(M)(p(x))=\prod_{i \in \Z}M(F^{i}(x))$. This makes it clear that $\pi$ preserves the simple modules. The following lemma was proved in \cite{Najera Frobenius orbit} using the fact that $\pi$ preserves projectives.

\begin{lemma}$($\cite[Lema 34]{Najera Frobenius orbit}$)$
\thlabel{embedding lemma}
\thlabel{projective resolution}
Let $M$ be a finitely presented $\ctt$-module. 
\begin{itemize}
\item[$(i)$] Let $L$ be an $\ctt$-module admitting a resolution $\cdots \to P_1 \to P_0 \to L \to  0$ by finitely generated projective $\ctt$-modules $P_i$. Then the complex
\begin{equation*}
\xymatrix{
\cdots\ar[r] & \pi (P_1) \ar[r] & \pi (P_0) \ar[r] &\pi (L) \ar[r] & 0
}
\end{equation*} 
is a resolution of $\pi (L)$ by finitely generated projective $\ctt\corb F_*$-modules;
\item[$(ii)$] for each $\ctt$-module $L$ admitting a resolution by finitely generated projective $\ctt$-modules, there are canonical isomorphisms
\begin{equation*}
\Ext^i_{\ctt \corb F_*}(\pi (L),\pi (M)) \cong  \prod_{l\in \Z} \Ext^i_{\ctt}( L, {F^l_*}_{\ast}(M)).
\end{equation*}
\end{itemize}
\end{lemma}

We can apply $\pi$ to the (co)resolutions of \thref{resolutions upstairs 1} and \thref{resolutions upstairs 2} to obtain the (co)resolutions of the simple $\ctt \corb F_*$-modules. For simplicity, we denote $\pi(x^{\wedge}_{\ctt})$ by $x^{\wedge}_{\ctt\corb F_*}$ and $\pi(x^{\vee}_{\ctt})$ by $x^{\vee}_{\ctt\corb F_*}$.

\begin{corollary}
\thlabel{resolution of simples 1}
Let $x$ be a vertex in $C$. Then the minimal projective resolution of the simple $\ctt\corb F_*$-module $S_{\sigma^{-1}(x))^{\wedge}}$ associated to the indecomposable object $p(\sigma^{-1}(x))^{\wedge})$ is
\begin{equation}
0 \to  {T_1^{x}}_{\ctt\corb F_{\ast}}^{\wedge} \to  {T_0^{x}}_{\ctt\corb F_{\ast}}^{\wedge}\oplus P_{x, \ctt\corb F_{\ast}} \to \sigma^{-1}(x)_{\ctt\corb F_{\ast}}^{\wedge} \to S_{p(\sigma^{-1}(x)^{\wedge})} \to 0,
\end{equation}
and its minimal injective coresolution is 
\begin{equation}
0\to S_{p(\sigma^{-1}(x)^{\wedge})}  \to x^{\vee}_{\ctt \corb F_{\ast}} \to I_{x,\ctt}\oplus {\Sigma T_0^x}^{\vee}_{\ctt \corb F_{\ast}} \to  {\Sigma T_1^x}^{\vee}_{\ctt \corb F_{\ast}} \to 0.
\end{equation}

\end{corollary}

\begin{corollary}
\thlabel{resolution of simples 2}
Let $x$ be a vertex in $\Z Q$ and suppose that $x^{\wedge}$ is an indecomposable object of $\ctt$. Then the minimal projective resolution of the simple $\ctt \corb F_*$-module $S_{p(x^{\wedge})}$ is 
\begin{equation}
0 \to x^{\wedge}_{\ctt\corb F_{\ast}} \to E'^{\wedge}_{\ctt\corb F_{\ast}} \to E^{\wedge}_{\ctt\corb F_{\ast}} \to x^{\wedge}_{\ctt\corb F_{\ast}} \to S_{p(x^{\wedge})} \to 0
\end{equation}
and its minimal injective coresolution is 
\[
0 \to S_{p(x^{\wedge})} \to x^{\vee}_{\ctt \corb F_{\ast}}\to E'^{\vee}_{\ctt \corb F_{\ast}} \to E^{\vee}_{\ctt \corb F_{\ast}} \to x^{\vee}_{\ctt \corb F_{\ast}} \to 0 .
\]
\end{corollary}

\subsection{A distinguished labeling of $\ind(\eafast)$} From now on we assume $Q$ is a bipartite orientation of a Dynkin diagram $\Delta$.
Let $W$ be the Weyl group associated to $\Phi^{\Delta}$ and $s_i\in W$ be the reflection associated to the simple root $\alpha_i$. Following \cite{associahedra} we define two bijections $\tau_{\pm}:\Phi_{\geq-1} \to \Phi_{\geq-1}$ by the formula
\begin{equation}
\label{eq:tau-action}
\tau_\varepsilon (\alpha) =
\begin{cases}
\alpha & \text{if $\alpha = -\alpha_j$ with $\varepsilon(j) = -\varepsilon$;} \\[.05in]
t_\varepsilon (\alpha)
 & \text{otherwise}
\end{cases}
\end{equation}
where
\begin{equation}
t_\varepsilon = \prod_{\varepsilon(k)=\varepsilon} s_k.
\end{equation}
These elements are well defined because the reflections associated to vertices having the same sign mutually commute. It follows from the same fact that each $t_\varepsilon$ is an involution which makes it evident that each $\tau_{\varepsilon}$ is bijective. 

\begin{notation}
Let $i$ be a vertex of $Q$. Denote by $P(i)$, $I(i)$ and $S(i)$ the projective, injective and simple module associated to $i$, respectively. We consider the objects of $\mod(kQ)$ as objects in $\cd_Q$ via the canonical embedding $\mod(kQ)\to \cd_Q$ which sends a module to a complex concentrated in degree 0.
\end{notation}

\begin{remark}
\thlabel{tau on roots}
Let $\ind(kQ)$ denote the set of indecomposable right $kQ$-modules and $\Phi_{>0}$ be the set of positive roots. The function $\ind(kQ)\to \Phi_{>0}$ given by
$$
M\mapsto \sum_{i=1}^n d_i\alpha_i,
$$
where $\dimv(M)=(d_1,\dots,d_n)$ is the dimension vector of $M$, is a bijection. We denote by $M_{\alpha}$ the indecomposable module corresponding to $\alpha\in \Phi_{>0}$. It is well known that the set of indecomposable object of $\cc_Q$ is the union of the set of indecomposable $kQ$-modules and the set $\lbrace \Sigma P(i) \rbrace^n_{i=1}$ formed by the indecomposable projective $kQ$-modules shifted by $\Sigma$. Therefore, the bijection $\ind(kQ)\overset{\sim}{\to} \Phi_{>0}$ can be extended to a bijection $\ind(\cc_Q)\to \Phi_{\geq -1}$ by sending the object $\Sigma P(i)=\tau P(i)$ to the negative simple root $-\alpha_i$. In particular, the suspension functor $\Sigma = \tau:\cc_Q \to \cc_Q$ induces a bijection (which we still call) $\tau: \Phi_{\geq-1} \to \Phi_{\geq-1}$. Moreover, it can be verified that $\tau_{-}\circ \tau_{+}=\tau$.
\end{remark}

For each vertex $i$ of $Q$ let $C^i:\mod(kQ)\to \mod(k\mu_i(Q))$ be the \emph{Bernstein-Gelfand-Ponomarev reflection functor} (\cf \cite{BGP}). The functor $C_i$ induces an autoequivalence 
\begin{equation*}
C^i_{\cd}:\cd_Q \to \cd_Q
\end{equation*}
sending the simple module $S(i)$ to $\Sigma^{-\varepsilon(i)}S(i)$. Let $C^{\varepsilon}_{\cd}$ be the composition of all the reflection functors $C^i_{\cd}$ such that $\varepsilon(i)=\varepsilon$. The following is a well known result.

\begin{lemma}
\thlabel{on coxeter functors}
The functors $C^{\varepsilon}_{\cd}$ satisfy the following properties:
\begin{itemize}
\item $C^{\varepsilon}_{\cd} \circ C^{\varepsilon}_{\cd}\cong Id$,
\item $\Sigma \circ C^{\varepsilon}_{\cd} \cong C^{\varepsilon}_{\cd}\circ \Sigma$,
\item $C_{\cd}^{\varepsilon}(P(i))=\Sigma^{-\varepsilon}I(i)$,
\item if $\alpha$ is not a simple root, then $C_{\cd}^{\varepsilon}(M_{\alpha})=M_{\tau_{\varepsilon}(\alpha)}$.
\end{itemize}
\end{lemma}

\begin{corollary}
\thlabel{tau and ext}
Let $M_{\alpha}$ be the indecomposable $kQ$-module corresponding to $\alpha\in \Phi_{>0}$. Let $i$ be a vertex of $Q$. If $i$ is a source then
\begin{equation*}
[\tau_{+}(\alpha):\alpha_i]=\dim \Ext_{kQ}^1(M_{\alpha},S(i)).
\end{equation*}
If $i$ is a sink of $Q$ then
\begin{equation*}
[\tau_{-}(\alpha):\alpha_i]=\dim \Ext_{kQ}^{-1}(M_{\alpha},S(i)).
\end{equation*}
\begin{proof}
We know that $[\tau_{\varepsilon}(\alpha):\alpha_i]=
\dim \Hom (M_{\tau_{\varepsilon}(\alpha)},I(i))$. By \thref{on coxeter functors} we have the following isomorphisms
\begin{align}
\Hom (M_{\tau_{\varepsilon}(\alpha)},I(i))&\cong \Hom (M_{\tau_{\varepsilon}(\alpha)}, \Sigma^\varepsilon C^{\varepsilon}_{\cd}(P(i))) \nonumber \\
&= \Hom (C^{\varepsilon}_{\cd} (M_{\tau_{\varepsilon}(\alpha)}), \Sigma^\varepsilon P(i)) \nonumber \\
&= \Hom (M_{\tau_{\varepsilon} \tau_{\varepsilon}(\alpha)}, \Sigma^\varepsilon P(i)). \nonumber \\
&=\Hom (M_{\alpha}, \Sigma^\varepsilon P(i)). \nonumber
\end{align}
If $i$ is a source then $P(i)=S(i)$. If $i$ is a sink then $P(i)=S(i)$. The claim follows.
\end{proof}
\end{corollary}

Let $C=\Z Q_0$ be the maximal admissible configuration (so $\cR_C=\cR$). We identify the objects of $\cc_Q$ with the non-projective objects of $\eafast$. In particular, the bijection $\ind(\cc_Q)\to \Phi_{\geq -1}$ of \thref{tau on roots} can be used to label the indecomposable non-projective objects of $\eafast$. Indeed, we denote by $X_{\alpha}$ the non-projective indecomposable object of $\eafast$ corresponding to $\alpha\in \Phi_{\geq-1}$ under the bijection. 
We can extend this labeling to the indecomposable projective objects of $\eafast$ as follows: let $P$ be an indecomposable projective object of $\eafast$. Then $P=p(y^{\wedge})$ for some frozen vertex $y$ of $\Z \widetilde{Q}$. 
The indecomposable non-projective object $p(\sigma(y)^{\wedge})$ is isomorphic to $X_{\alpha}$ for some $\alpha \in \Phi_{\geq -1}$. We label $P$ by $\tau_+(\alpha)$ and write $P_{\tau_+(\alpha)}$. Note that the indecomposable projective object $p(\sigma^{-1}(x)^{\wedge})$ is labeled by $\tau_-(\alpha)$ because $\tau_-(\alpha)=\tau_+\circ \tau^{-1}(\alpha)$.
\begin{example}
Let $Q: 1\to 2 \leftarrow 3$ be a bipartite quiver of type $A_3$. There are two $\tau$-orbits in $\Phi_{\geq -1}$ that can be visualized as follows:
\begin{equation*}
\xymatrix{
-\alpha_{1} \ar@(dr,dl)^-{\tau_-}[] \ar@{<->}^-{\tau_+}[r] & \alpha_1  \ar@{<->}^-{ \tau_-}[r] &  \alpha_1 + \alpha_2 \ar@{<->}^-{\tau_+}[r] & 
 \alpha_2 + \alpha_3 \ar@{<->}^-{\tau_-}[r] & \alpha_3 \ar@{<->}^-{\tau_+ }[r] & -\alpha_3 \ar@(dl,dr)_-{\tau_-}[] 
}
\end{equation*}
\begin{equation*}
\xymatrix{
-\alpha_2  \ar@{<->}^-{\tau_-}[r] \ar@(dr,dl)^-{\tau_+}[] &  \alpha_2 \ar@{<->}^-{\tau_+}[r] &
\alpha_1 + \alpha_2 + \alpha_3. \ar@(dl,dr)_-{\tau_-}[] 
}
\end{equation*}
In Figure \ref{A_3 universal labeled} below we have represented the AR quiver of the category $\eafast$. Its vertices are labeled by the rule described above.
\begin{figure}[htbp]
\[
\xymatrix@=0.35cm@!{
X_{\alpha_{1}}\ar[r] \ar@{--}[d] \ar[dr]& P_{\alpha_1+\alpha_2} \ar[r]  & X_{\alpha_2 + \alpha_3} \ar[r]  \ar[dr]  & P_{\alpha_3} \ar[r] & X_{-\alpha_3} \ar[r] \ar[dr] & P_{-\alpha_3} \ar[r] &X_{\alpha_3}   \ar@{--}[d]   \\
P_{\alpha_2} \ar[r] \ar@{--}[d]& X_{\alpha_1 + \alpha_2 + \alpha_3} \ar[dr]  \ar[r] \ar[ur] & P_{\alpha_1+\alpha_2+\alpha_3} \ar[r] & X_{\alpha_2} \ar[r] \ar[ur] \ar[dr] & P_{-\alpha_2} \ar[r] & X_{-\alpha_2}\ar[dr] \ar[r] \ar[ur] & P_{\alpha_2} \ar@{--}[d] \\
 X_{\alpha_3}  \ar[r] \ar[ur] & P_{\alpha_2 + \alpha_3} \ar[r] & X_{\alpha_1+\alpha_2} \ar[r] \ar[ur] & P_{\alpha_1} \ar[r] & X_{-\alpha_1} \ar[r] \ar[ur] & P_{-\alpha_1} \ar[r] &  X_{\alpha_1}&
  }
\]
\caption{The quiver of $\ind(\eafast)$ for $Q$ of type $A_3$.}
\label{A_3 universal labeled}
\end{figure}
\end{example}

\subsection{Proof of \thref{categorification}} We are ready to prove our second main result. Let $C=\Z Q_0$ be the maximal admissible configuration. Throughout this subsection we let $T$ be the object $\bigoplus_{i\in Q_0} X_{-\alpha_i} \in \cc_Q$. Clearly, $T$ is a cluster-tilting object whose quiver is $Q$. Let $\ct / F$ be the associated cluster-tilting subcategory of $\cc_Q$ (see \thref{cts}). Note that the exchange triangles associated to $X_{-\alpha_i}$ and $\ct$ are of the form
\\
\begin{center}
\begin{tabular}{lllll}
$\phantom{l} X_{\alpha_i} \longrightarrow B \longrightarrow X_{-\alpha_i} \to \Sigma X_{\alpha_i}$
&
&
and
&
&
$ X_{-\alpha_i} \longrightarrow B' \longrightarrow X_{\alpha_i} \to \Sigma X_{-\alpha_i}$.
\end{tabular}
\end{center}

\begin{proposition}
\thlabel{exch sequences general}
Let $i$ be a vertex of $Q$.
\begin{itemize}
\item[$(i)$] If $i$ is a source then  the exchange conflations associated to $X_{-\alpha_i}$ and $\ctt \corb F_*$ have the form:
\begin{equation}
\label{first ses}
0 \longrightarrow X_{-\alpha_i} \longrightarrow   \bigoplus_{i\to j}X_{-\alpha_j} \oplus  P_{-\alpha_i}  \longrightarrow X_{\alpha_i} \longrightarrow 0
\end{equation}
and
\begin{equation}
\label{second ses}
0 \longrightarrow X_{\alpha_i} \longrightarrow \bigoplus_{\alpha \in \Phi_{>0}} P^{[\alpha:\alpha_i]}_{\alpha} \longrightarrow X_{-\alpha_i} \longrightarrow 0.
\end{equation}
\end{itemize}
\begin{itemize}
\item[$(ii)$] If $i$ is a sink then  the exchange conflations associated to $X_{-\alpha_i}$ and $\ctt \corb F_*$ have the form:
\begin{equation*}
0 \longrightarrow X_{-\alpha_i} \longrightarrow\bigoplus_{\alpha \in \Phi_{>0}} P^{[\alpha:\alpha_i]}_{\alpha} \longrightarrow X_{\alpha_i} \longrightarrow 0
\end{equation*}
and
\begin{equation*}
0 \longrightarrow X_{\alpha_i} \longrightarrow \bigoplus_{j\to i}X_{-\alpha_j} \oplus P_{-\alpha_i} \longrightarrow X_{-\alpha_i} \longrightarrow 0.
\end{equation*}
\end{itemize}
\end{proposition}
\begin{proof}
Let $i$ be a source of $Q$ and $x\in (\Z \widetilde{Q})_0$ be a non-frozen vertex such that 
\begin{equation*}
p(x^{\wedge})=X_{-\alpha_i}.
\end{equation*} 
In particular, 
\begin{equation*}
p(\sigma^{-1}(x)^{\wedge})=P_{\tau_-(-\alpha_i)}=P_{-\alpha_i}
 \mbox{  \ \ \   and \ \ \ }
 p(\tau^{-1}(x)^{\wedge})=X_{\alpha_i}. 
\end{equation*}
The mesh relation starting at $x$ gives a conflation
\begin{equation*}
0 \to x^{\wedge} \to \bigoplus_{x\to y} y^{\wedge} \oplus (\sigma^{-1}(x))^{\wedge} \to (\tau^{-1}(x))^{\wedge}\to 0
\end{equation*}
in $\ce$, where the sum ranges over the set of arrows of $\Z \widetilde{Q}$ whose tail is $x$ and whose head is a non-frozen vertex $y$. By the exactness of $p$, there is a conflation
\begin{equation}
\label{mesh relation on proof}
0 \to X_{-\alpha_i} \to \bigoplus_{x\to y} p(y^{\wedge}) \oplus P_{-\alpha_i} \to X_{\alpha_i}\to 0
\end{equation}
in $\eaf$. By construction, the image of \eqref{mesh relation on proof} in the stable category $\ul{\eafast} \cong \cd_Q/F$ is the image of the mesh relation in $\cd_Q$ starting at $\Sigma P(i)$. Thus 
\begin{equation*}
 \bigoplus_{x\to y} p(y^{\wedge}) =  \bigoplus_{i\to j\in Q_1} \Sigma P(j).
 \end{equation*}
Since $ \Sigma P(j)$ corresponds to $-\alpha_j$ under the bijection $\ind(\cc_Q)\to \Phi_{\geq -1}$ the sequence \eqref{mesh relation on proof} identifies with the sequence \eqref{first ses}. 

Since $i$ is a source, there is a triangle
\[
X_{\alpha_i}\to 0 \to X_{-\alpha_i}\to X_{-\alpha_i}
\]
in $\cc_Q$. Therefore, in this case, the exact sequence of $\cR$-modules in (\ref{resolution of x D/T}) is given by
\begin{equation*}
0 \rightarrow (\Sigma^{-1}x)^{\wedge} \rightarrow P_{x,\ct} \rightarrow x^{\wedge} \rightarrow x^{\wedge}_{\cd/\ct} \rightarrow 0,
\end{equation*}
and gives rise to the conflation
\begin{equation}
\label{ses in eaf}
0\to {\Sigma^{-1}x}^{\wedge}_{\ce} \to P_{x,\ct} \to x^{\wedge}_{\ce}\to 0
\end{equation}
in $\ce$. We apply the canonical projection $p:\ce \to \eafast$ to this conflation to a obtain the conflation
\begin{equation}
\label{ses in eaf}
0\to \Sigma^{-1}X_{-\alpha_i} \to p(P_{x,\ct}) \to X_{-\alpha_i}\to 0
\end{equation}
of $\eafast$. We claim that \eqref{second ses} corresponds to \eqref{ses in eaf}. Indeed: since $C=\Z Q_0$ and
\begin{equation*}
\Sigma^{-1}p (x^{\wedge})= \tau^{-1}p (x^{\wedge})
\end{equation*}
in $\eafast$, then \eqref{ses in eaf} takes the form
\begin{equation*}
0 \rightarrow X_{\alpha_i} \rightarrow  \bigoplus_{y\in \Z Q} \cd_Q(y,x) \ten p(\sigma(y)^\wedge)\rightarrow X_{-\alpha_i} \rightarrow 0.
\end{equation*}
We have seen that $x=\Sigma P(i)=\Sigma S(i)$. We obtain the following isomorphisms
\begin{align}
\bigoplus_{y\in \Z Q} \cd_Q(y,x)&=\bigoplus_{\alpha \in \Phi_{>0}} \cd_Q(M_{\alpha}, \Sigma S(i)) \nonumber \\
&= \bigoplus_{\alpha \in \Phi_{>0}} \Ext_{kQ}^1(M_{\alpha}, S(i)).
\end{align}
We obtain that the multiplicity of $P_{\alpha}$ in $P_{x, \ct}$ is equal to the dimension of $\Ext^1_{kQ}(M_{\tau_{+}(\alpha)},S_i)$. 
 By \thref{tau and ext} we know that $\dim \Ext^1_{kQ}(M_{\tau_{+}(\alpha)},S_i)=[\alpha, \alpha_i]$, \ie
 \begin{equation*}
p(P(x)) = \bigoplus_{\alpha \in \Phi_{>0}} P^{[\alpha:\alpha_i]}_{\alpha}.  
 \end{equation*}
 The claim follows. We can treat similarly the case where $i$ is a sink.
\end{proof}
\begin{proof}[Proof of \thref{categorification}] $(i)$ Let $C$ be an admissible configuration. Part $(i)$ states that $\ce/F_{\ast}$ is a Frobenius 2-CY realization of a cluster algebra with coefficients of type $Q$. This is a direct consequence of Theorem II.1.6 of \cite{BIRS}. 

$(ii)$ In view of \thref{a remark on frozen quivers}, it is enough to prove that there is a cluster-tilting object of $ \eafast$ whose frozen quiver coincides with the initial coefficients defining $\ca_{Q}^{\rm univ}$ described in \thref{th universal coefficients}. We claim that
\begin{equation*}
\widetilde{T}=\bigoplus_{i\in Q_0} X_{-\alpha_i} \oplus \bigoplus_{\alpha \in \Phi_{\geq -1}}P_{\alpha}
\end{equation*}
satisfies this assertion. It is clear that $\widetilde{T}$ is a cluster-tilting object since $\bigoplus_{i\in Q_0} X_{-\alpha_i}$ is a cluster-tilting object of $\cc_Q$ . Let $\ct = \add(T)$. Denote by $i$ the vertex of $Q^{F}_{\ct}$ corresponding to the indecomposable object $X_{-\alpha_i}$ and by $\boxed{\alpha}$ the frozen vertex of $Q^F_{\ct}$ corresponding to $P_{\alpha}$, $\alpha \in \Phi_{\geq -1}$. We can describe $Q^{F}_{\ct}$ using the projective resolutions of the simple $\eafast$-modules described in \thref{resolution of simples 1} and \thref{resolution of simples 2}. To compute these resolutions it is enough to describe the exchange conflations associated to $\widetilde{T}$ given in \thref{exch sequences general}. If $i$ is a source of $Q$ then there are
\begin{itemize}
\item $[\alpha:\alpha_i]$ arrows from $\boxed{\alpha}$ to $i$ for each $\alpha\in \Phi_{>0}$,
\item one arrow from $i$ to $j$ for each arrow $i\to j $ in $Q$,
\item one arrow from $i$ to the vertex corresponding to $\boxed{-\alpha_i}$.
\end{itemize}
This is precisely the frozen quiver encoding by the initial coefficients of $\ca_{Q}^{\rm univ}$ described in \thref{th universal coefficients}.
We obtain a similar description when $i$ is a sink of $Q$. The claim follows.
\end{proof}

\section{Examples}

\begin{example}
Let $Q:1\to 2$ be a Dynkin quiver of type $A_2$.
Thus, $\varepsilon(1)=+$ and $\varepsilon(2)=-$.
By \thref{th universal coefficients}, the cluster algebra with universal coefficients of type $A_2$
is defined over the tropical semifield
\begin{equation*}
\label{eq:PP-univ-A2}
\P = \Trop
(p_{-\alpha_1},p_{-\alpha_2},p_{\alpha_1},p_{\alpha_1+\alpha_2},p_{\alpha_2}).
\end{equation*}
The cluster algebra $\ca_{A_2}^{\rm univ}$ is completely determined by the ice quiver
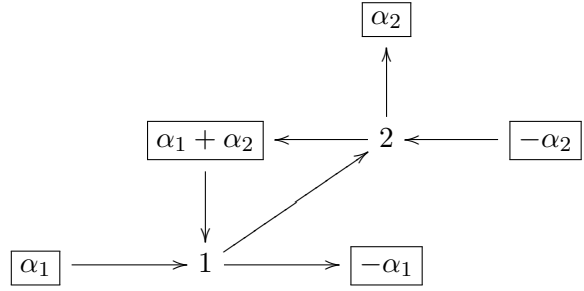
\begin{equation*}
\xymatrix{
& & \boxed{\alpha_2} & \\
& \boxed{\alpha_1 + \alpha_2} \ar[d] & 2 \ar[l] \ar[u] & \boxed{-\alpha_2} \ar[l] \\
\boxed{\alpha_1} \ar[r] & 1\ar[ru] \ar[r] & \boxed{-\alpha_1} & 
}
\end{equation*}
where the frozen vertex corresponding to the generator $\alpha\in \Phi_{\geq -1}$ is denoted by $\boxed{\alpha}$.

Recall that there is a bijection between the cluster variables in $\ca_{A_2}^{\rm univ}$ and the roots in $\Phi_{\geq-1}$. We let 
$x_{\alpha}$ denote the cluster variable associated to $\alpha$. Then the exchange relations in~$\ca_{A_2}^{\rm univ}$
are:
\begin{align}
\label{eq:exch-univ-A2-1}
x_{-\alpha_2}\,x_{\alpha_2} \nonumber
&= p_{-\alpha_2}\,x_{-\alpha_1}+p_{\alpha_2}\,p_{\alpha_1+\alpha_2}\,,\\[.1in] \nonumber
x_{-\alpha_1}\,x_{\alpha_1+\alpha_2} \nonumber
&= p_{\alpha_1}\,x_{\alpha_2}+ p_{-\alpha_1}\,p_{\alpha_2}\,,\\[.1in]
x_{\alpha_2}\,x_{\alpha_1} \nonumber
&= p_{\alpha_1+\alpha_2}\,x_{\alpha_1+\alpha_2}+ p_{-\alpha_1}\,p_{-\alpha_2}\,,\\[.1in]
x_{\alpha_1+\alpha_2}\,x_{-\alpha_2} \nonumber
&= p_{\alpha_2}\,x_{\alpha_1}+ p_{-\alpha_2}\,p_{\alpha_1}\,,\\[.1in]
\nonumber
x_{\alpha_1}\,x_{-\alpha_1}
&= p_{-\alpha_1}\,x_{-\alpha_2}+p_{\alpha_1}\,p_{\alpha_1+\alpha_2}\,.
\end{align}

In this case, the quiver of $\ce\corb F_*$ is the following
$$
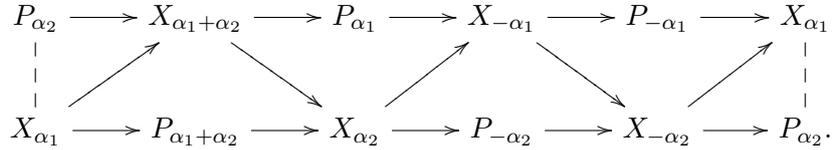
\xymatrix{
P_{\alpha_2} \ar[r] \ar@{--}[d]  &X_{\alpha_1 + \alpha_2} \ar[r] \ar[rd] & P_{\alpha_1} \ar[r] & X_{-\alpha_1} \ar[r] \ar[rd] & P_{-\alpha_1} \ar[r] &  X_{\alpha_1} \ar@{--}[d] \\
X_{\alpha_1} \ar[r] \ar[ru] & P_{\alpha_1+\alpha_2}  \ar[r] & X_{\alpha_2} \ar[r] \ar[ru] &  P_{-\alpha_2} \ar[r] &  X_{-\alpha_2}\ar[r] \ar[ru] &   P_{\alpha_2}.
}
$$
We can verify that the exchange relations above correspond to the exchange conflations in $\ce \corb F$. For instances, the exchange relation 
\begin{equation*}
x_{\alpha_2}\,x_{\alpha_1} 
= p_{\alpha_1+\alpha_2}\,x_{\alpha_1+\alpha_2}+ p_{-\alpha_1}\,p_{-\alpha_2}
\end{equation*}
corresponds to the non-split short exact sequences
\begin{equation*}
0 \longrightarrow X_{\alpha_1} \longrightarrow  P_{\alpha_1+\alpha_2}\oplus X_{\alpha_1+\alpha_2}\longrightarrow X_{\alpha_2} \longrightarrow 0
\end{equation*}
and
\begin{equation*}
0 \longrightarrow X_{\alpha_2} \longrightarrow P_{-\alpha_1}\oplus P_{-\alpha_2} \longrightarrow X_{\alpha_1} \longrightarrow 0.
\end{equation*}
The ice quiver of
$
T=X_{-\alpha_1} \oplus X_{-\alpha_2} \oplus P_{-\alpha_1} \oplus P_{-\alpha_2} \oplus P_{\alpha_1} \oplus P_{\alpha_1+\alpha_2} \oplus P_{\alpha_2} 
$ is
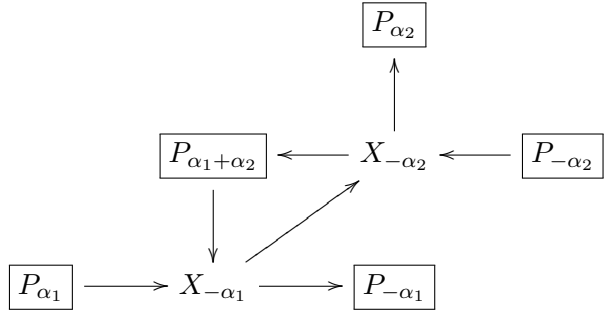
\begin{equation*}
\xymatrix{
& & \boxed{P_{\alpha_2}} & \\
& \boxed{P_{\alpha_1 + \alpha_2}} \ar[d] & X_{-\alpha_2} \ar[l] \ar[u] & \boxed{P_{-\alpha_2}} \ar[l] \\
\boxed{P_{\alpha_1}} \ar[r] & X_{-\alpha_1}\ar[ru] \ar[r] & \boxed{P_{-\alpha_1}} & 
}
\end{equation*}
as expected.
\end{example}

Using \thref{Frobenius model orbit} we can see that the categorification of finite-type Grassmannian cluster algebras introduced in \cite{JKS} is equivalent to $\gpr(\cs_C)\corb F_*$ for an appropriate choice of admissible configuration $C$. This follows from the fact that these categories are standard Frobenius models of a finite-type cluster category.

\end{document}